\pdfoutput=1
\RequirePackage{ifpdf}
\ifpdf 
\documentclass[pdftex]{sigma}
\else
\documentclass{sigma}
\fi

\numberwithin{equation}{section}

\newtheorem{Theorem}{Theorem}[section]
\newtheorem{Corollary}[Theorem]{Corollary}
\newtheorem{Lemma}[Theorem]{Lemma}
\newtheorem{Proposition}[Theorem]{Proposition}
 { \theoremstyle{definition}
\newtheorem{Definition}[Theorem]{Definition}

\newtheorem{Remark}[Theorem]{Remark} }

\usepackage{ytableau}

\begin{document}

\newcommand{\arXivNumber}{1912.06488}

\renewcommand{\PaperNumber}{031}

\FirstPageHeading

\ShortArticleName{Representations of the Lie Superalgebra $\mathfrak{osp}(1|2n)$ with Polynomial Bases}

\ArticleName{Representations of the Lie Superalgebra $\boldsymbol{\mathfrak{osp}(1|2n)}$\\ with Polynomial Bases}

\Author{Asmus K.~BISBO~$^{\rm a}$, Hendrik DE BIE~$^{\rm b}$ and Joris VAN DER JEUGT~$^{\rm a}$}

\AuthorNameForHeading{A.K.~Bisbo, H.~De Bie and J.~Van der Jeugt}

\Address{$^{\rm a)}$~Department of Applied Mathematics, Computer Science and Statistics, Ghent University,\\
\hphantom{$^{\rm a)}$}~Krijgslaan 281-S9, B-9000 Gent, Belgium}
\EmailD{\href{mailto:Asmus.Bisbo@UGent.be}{Asmus.Bisbo@UGent.be}, \href{mailto:Joris.VanderJeugt@UGent.be}{Joris.VanderJeugt@UGent.be}}

\Address{$^{\rm b)}$~Department of Electronics and Information Systems, Faculty of Engineering and Architecture,\\
\hphantom{$^{\rm b)}$}~Ghent University, Krijgslaan 281-S8, B-9000 Gent, Belgium}
\EmailD{\href{mailto:Hendrik.DeBie@Ugent.be}{Hendrik.DeBie@Ugent.be}}

\ArticleDates{Received June 30, 2020, in final form March 10, 2021; Published online March 25, 2021}

\Abstract{We study a particular class of infinite-dimensional representations of~$\mathfrak{osp}(1|2n)$. These representations $L_n(p)$ are characterized by a positive integer $p$, and are the lowest component in the $p$-fold tensor product of the metaplectic representation of~$\mathfrak{osp}(1|2n)$. We~construct a new polynomial basis for~$L_n(p)$ arising from the embedding $\mathfrak{osp}(1|2np) \supset \mathfrak{osp}(1|2n)$. The basis vectors of~$L_n(p)$ are labelled by semi-standard Young tableaux, and are expressed as Clifford algebra valued polynomials with integer coefficients in $np$ variables. Using combinatorial properties of these tableau vectors it is deduced that they form indeed a basis. The~computation of matrix elements of a set of generators of~$\mathfrak{osp}(1|2n)$ on these basis vectors requires further combinatorics, such as the action of a Young subgroup on the horizontal strips of the tableau.}


\Classification{17B10; 05E10; 81R05; 15A66}

\section{Introduction}\label{sec_introduction}

In representation theory of Lie algebras, Lie superalgebras or their deformations, there are often three problems to be tackled.
The first is the existence or the classification of representations.
The second is obtaining character formulas for representations.
And the third is the construction of (a class of) representations.
Concretely, this third step consists of finding an explicit basis for the representation space and the explicit action of a set of algebra generators in this basis (i.e., find all matrix elements).
Mathematicians are primarily interested in the first two problems and often ignore the third one.
This goes together with the impression that the third problem is computationally quite hard and not necessarily leads to interesting mathematical structures.
For~applications in physics however, the third step is often indispensable, as one needs to compute physical quantities such as energy spectra coming from eigenvalues of Hamiltonians, or transition matrix elements coming from explicit actions, e.g.,~\cite{Iachello-VanIsacker-1991}. A~typical example of the third problem is the construction of the Gelfand--Zetlin basis for finite-dimensional irreducible representations of the Lie algebra $\mathfrak{gl}(n)$ or $\mathfrak{sl}(n)$, see, e.g.,~\cite{Molev-2006}.

In this paper we are dealing with a class of representations of the orthosymplectic Lie super\-algebra $\mathfrak{osp}(1|2n)$.
The representations considered here are infinite-dimensional irreducible lowest weight representations, with lowest weight coordinates
$\big(\frac{p}{2}, \frac{p}{2}, \ldots, \frac{p}{2}\big)$, for~$p$ a positive inte\-ger~\cite{Blank-Havlicek-1986, Dobrev-Salom-2017, Dobrev-Zhang-2005}. As~we shall see in the paper, the construction of basis vectors and generator actions leads to many interesting mathematical and combinatorial concepts.

The Lie superalgebra $\mathfrak{osp}(1\vert 2 n)$ plays an important role in mathematical physics.
For $n=2$, this superalgebra describes the supersymmetric extension in $D=4$ anti-de Sitter space (\cite{Deser-Zumino-1977},~see also~\cite{Ivanov-Sorin-1980} and references therein).
Supersymmetric higher spin extensions of the anti-de Sitter algebra in four dimensions are realized in other orthosymplectic superalgebras, including $\mathfrak{osp}(1\vert 2 n)$ and~$\mathfrak{osp}(m\vert 2 n)$~\cite{Sezgin-Sundell-2013, Vasiliev-2002}.
Our own motivation for studying the class of~$\mathfrak{osp}(1\vert 2 n)$ representations mentioned in the previous paragraph also comes from physics:
this class of representations corresponds to the so-called paraboson Fock spaces.
Parabosons were introduced by Green~\cite{Green-1953} in~1953, as generalizations of bosons.
Parabosons have been of interest in quantum field theory~\cite{Ohnuki-Kamefuchi-1982},
in generalizations of quantum statistics~\cite{Chaturvedi-1996, Greenberg-Messaih-1965} and in Wigner quantum systems~\cite{King-Palev-Stoilova-VanderJeugt-2003,Lievens-Stoilova-VanderJeugt-2006}.
Whereas creation and annihilation operators of bosons satisfy simple commutation relations, those of parabosons satisfy more complicated triple relations.
Moreover where there is only one boson Fock space, there are an infinite number of paraboson Fock spaces, each of them characterized by a positive integer $p$
(called the order of statistics).
Many years after their introduction, it was shown that the triple relations for~$n$ pairs of paraboson operators are in fact defining relations for the Lie superalgebra $\mathfrak{osp}(1|2n)$~\cite{Ganchev-Palev-1980}
and that the paraboson Fock space of order $p$ coincides with the unitary irreducible lowest weight representation~$L_n(p)$ of~$\mathfrak{osp}(1|2n)$ with lowest weight
$\big(\frac{p}{2}, \frac{p}{2}, \ldots, \frac{p}{2}\big)$ in the natural basis of the weight space.

The construction of a convenient basis for~$L_n(p)$, with the explicit action of the paraboson operators, turns out to be a difficult problem. In~principle, one can follow Green's approach \cite{Green-1953, Kanakoglou-Daskaloyannis-2007} and identify $L_n(p)$ as an irreducible component in the $p$-fold tensor product
of the boson Fock space of~$\mathfrak{osp}(1|2n)$ (which is $L_n(1)$).
However, there are computational difficulties to finding a~proper basis this way.

Some years ago~\cite{Lievens-Stoilova-VanderJeugt-2008}, a solution was found for the construction of a basis for~$L_n(p)$, \mbox{using} in particular the embedding $\mathfrak{osp}(1|2n)\supset\mathfrak{gl(n)}$.
This allowed the construction of a proper Gelfand--Zetlin basis for~$L_n(p)$, and using further group theoretical techniques the actions of~the
paraboson operators in this basis could be computed~\cite{Lievens-Stoilova-VanderJeugt-2008}.
This offered a solution to a~long-standing problem. A~disadvantage of this solution and the Gelfand--Zetlin basis is the rather complicated expressions for the generator matrix elements, involving square roots and elaborate Clebsch–Gordan coefficients. For the basis constructed in this paper, all generator matrix elements are integers, calculated through combinatorial considerations.
Moreover, the individual basis elements of our construction are described by concrete polynomial expressions.
This is another difference with the Gelfand--Zetlin basis, which is essentially just a labelling of~basis vectors without a concrete functional description.

In the present paper, we construct a new polynomial basis for these representations $L_n(p)$.
Starting from the embedding $\mathfrak{osp}(1|2np) \supset \mathfrak{osp}(1|2n)$, $L_n(p)$ can be identified with a sub\-mo\-dule
of the decomposition of the $\mathfrak{osp}(1|2np)$ Fock space $L_{np}(1)$.
Equivalently, the Howe dual pair $(\mathfrak{osp}(1|2n), \mathop{\rm Pin}(p))$ for this Fock space can be employed.
Using furthermore the character for~$L_n(p)$, this yields a polynomial basis consisting of vectors $\omega_A(p)$, labelled by semi-standard Young tableaux $A$ of length at most $p$ with entries $1,2,\ldots,n$.
The construction of this basis is carefully developed in Sections~\ref{sec2}--\ref{sec5}.
Each basis vector $\omega_A(p)$ is a specific Clifford algebra valued polynomial with integer coefficients in $np$ variables $x_{i,\alpha}$ ($i=1,\ldots,n$; $\alpha=1,\ldots,p$), thus invol\-ving $p$ Clifford elements $e_\alpha$.
The $\mathfrak{osp}(1|2n)$ generators $X_i$ and~$D_i$ are realized as (Clifford algebra valued) multiplication operators or differentiation operators with respect to the $x_{i,\alpha}$'s.
These are given in Section~\ref{sec2}, where the definition of~$\mathfrak{osp}(1|2n)$ is recalled and~$L_n(p)$ is defined. In~Section~\ref{sec3} a module $\overline{V}_n(p)$ is constructed, which is a natural induced $\mathfrak{osp}(1|2n)$ module. A~useful homomorphism $\Psi_p\colon\overline{V}_n(p)\to L_n(p)$ is considered, yielding the conclusion that~$L_n(p)$ is a~quotient module of~$\overline{V}_n(p)$. In~Section~\ref{sec4} we introduce the above mentioned tableau vectors~$\omega_A(p)$ as candidate basis vectors for~$L_n(p)$.
The same is done for tableau vectors~$v_A(p)$ of~the module~$\overline{V}_n(p)$. In~this section, the knowledge of the characters of~$L_n(p)$ and~$\overline{V}_n(p)$ in~terms of Schur polynomials plays an essential role. In~order to show that the tableau vectors~$\omega_A(p)$ constitute indeed a basis for~$L_n(p)$, their linear independence must be shown,
and this is established in~Section~\ref{sec5}.
The~proof is non-trivial and depends on a total ordering for the set of semistandard Young tableaux with entries from $\{1,2,\ldots,n\}$.
This allows the identification of a unique ``leading term'' in $\omega_A(p)$ which does not appear in $\omega_B(p)$ if $B<A$. In~Section~\ref{sec6} we compute the action of the $\mathfrak{osp}(1|2n)$ generators $X_i$ and~$D_i$ on the tableau vectors $\omega_A(p)$ of~$L_n(p)$.
This is rather technical: rewriting $X_i\omega_A(p)$ or $D_i \omega_A(p)$ as linear combinations of tableau vectors $\omega_B(p)$ is not trivial.
Fortunately, also here the identification of ``leading terms'' is very helpful to solve the problem. We~also give some examples, making the technical parts comprehensible. Finally, in Section~\ref{sec7} we consider as an example the case $n=2$ in detail.

To improve the readability of the paper, we make a list of the commonly used symbols. Here $n,p\in\mathbb{N}$ and~$i\in\{1,\dots,n\}$.
\begin{gather*}
\begin{split}
\mathcal{H}_n(p)&\ \text{ Irreducible lowest weight $\mathfrak{osp}(1|2n)$-module with lowest weight $\big(\tfrac{p}{2},\ldots,\tfrac{p}{2}\big)$}\\
\mathcal{C}\ell_p&\ \text{ Complex Clifford algebra with $p$ positive signature generators, $e_\alpha$}\\
\mathbb{C}[\mathbb{R}^{np}]&\ \text{ Space of polynomials in $n p$ variables, $x_{i,\alpha}$, with complex coefficients}\\
\mathcal{A}&\ \text{ Space of Clifford algebra valued polynomials, $\mathcal{A}:=\mathbb{C}[\mathbb{R}^{np}]\otimes \mathcal{C}\ell_p$}\\
L_n(p)&\ \text{ Realization of~$\mathcal{H}_n(p)$ in $\mathcal{A}$ generated by the $\mathfrak{osp}(1|2n)$-action on~$1\in \mathcal{A}$}\\
X_i,D_i&\ \text{ Generators of the $\mathfrak{osp}(1|2n)$-action on~$L_n(p)$ and~$\mathcal{A}$}\\
\overline{V}_n(p)&\ \text{ Induced $\mathfrak{osp}(1|2n)$-module with lowest weight vector $|0\rangle$ of weight $\big(\tfrac{p}{2},\ldots,\tfrac{p}{2}\big)$}\\
B_i^+,B_i^-&\ \text{ Generators of the $\mathfrak{osp}(1|2n)$-action on~$\overline{V}_n(p)$}\\
M_n(p)&\ \text{ Maximal non-trivial $\mathfrak{osp}(1|2n)$-submodule of~$\overline{V}_n(p)$}\\
V_n(p)&\ \text{ Realization of~$\mathcal{H}_n(p)$ as the quotient module $V_n(p):=\overline{V}_n(p)/M_n(p)$}\\
\Psi_p&\ \text{ The map $\overline{V}_n(p)\to L_n(p)$ mapping $B_{i_1}^+\cdots B_{i_k}^+|0\rangle \mapsto X_{i_1}\cdots X_{i_k}(1)$}\\
\mathbb{Y}_n&\ \text{ Set of semistandard Young tableaux with entries in $\{1,\dots,n\}$}\\
\mathbb{Y}_n(p)&\ \text{ Set of semistandard Young tableaux in $\mathbb{Y}_n$ with at most $p$ rows}\\
\mathbb{T}(\lambda,p)&\ \text{ Set of column distinct Young tableaux of shape $\lambda$ with entries in $\{1,\dots,p\}$}\\
\mathbb{E}(\lambda,p)&\ \text{ Set of Young tableaux of shape $\lambda$ with entries in $\{1,\dots,p\}$}\\
\lambda_A, \mu_A&\ \text{ Respectively the shape and weight of a tableau $A\in\mathbb{Y}_n$}\\
\omega_A(p)&\ \text{ Basis vector of~$L_n(p)$ corresponding to $A\in\mathbb{Y}_n(p)$}\\	
\tilde{\omega}_A(p)&\ \text{ Normalized vector propotional to $\omega_A(p)$}\\
v_A(p)&\ \text{ Basis vector of~$\overline{V}_n(p)$ corresponding to $A\in\mathbb{Y}_n$}\\
x_{A,T}&\ \text{ Monomial in $\mathbb{C}[\mathbb{R}^{np}]$ corresponding to $A\in\mathbb{Y}_n(p)$ and~$T\in\mathbb{E}(\lambda,p)$}\\
e_T&\ \text{ Element in $\mathcal{C}\ell_p$ corresponding to $T\in\mathbb{E}(\lambda,p)$}.\\
D_A&\ \text{ Young tableau of shape $\lambda_A$ with $k$'s in all entries of the $k$'th row, for all $k$}\\
c_A(\gamma)&\
\text{ Coefficient of the monomial $\textstyle{\prod_{i=1}^n\prod_{\alpha=1}^p (x_{i,\alpha}e_\alpha)^{\gamma_{i,\alpha}}}$ in the expansion of~$\tilde{\omega}_A(p)$}
	\end{split}
	\end{gather*}

\section[The polynomial paraboson Fock space Ln(p)]
{The polynomial paraboson Fock space $\boldsymbol{L_n(p)}$}\label{sec2}

The Lie superalgebra $\mathfrak{osp}(1|2n)$ is usually defined as a matrix Lie superalgebra \cite{Frappat-Sorba-Sciarrino-2000, Kac-1977}. It can also be defined as a symbolic algebra with generators and relations \cite{Ganchev-Palev-1980}.
Adopting the latter definition~$\mathfrak{osp}(1|2n)$ is generated by $2n$ odd elements $B^+_i$ and~$B^-_i$, for~$i\in\{1,\dots,n\}$, satisfying the structural relations
\begin{gather}
\label{sec2_eq_PCR}
\big[\big\{B_i^\xi,B_j^\eta\big\},B_l^\epsilon\big]=(\epsilon-\xi)\delta_{i,l}B_j^\eta
+(\epsilon-\eta)\delta_{j,l}B_i^\xi,
	\end{gather}
for~$i,j,l\in \{1,\dots,n\}$ and~$\eta,\epsilon,\xi\in\{+,-\}$, to be interpreted as $\pm 1$ in the algebraic relations. The brackets $[\cdot,\cdot]$ and~$\{\cdot,\cdot\}$ in~\eqref{sec2_eq_PCR} denote commutators and anti-commutators respectively. We~endow $\mathfrak{osp}(1|2n)$ with a unitary structure given by the anti-involution~$B_i^\pm\mapsto (B_i^\pm)^\dagger:=B_i^\mp$.

The Cartan subalgebra $\mathfrak{h}$ of~$\mathfrak{osp}(1|2n)$ has a basis consisting of the $n$ commuting elements
\begin{gather}
\label{sec2_eq_cartan}
h_i=\frac{1}{2}\big\{B^+_i,B^-_i\big\},
	\end{gather}
for~$i\in\{1,\dots,n\}$. Letting $\epsilon_i$, for~$i\in\{1,\dots,n\}$, be the corresponding dual basis for~$\mathfrak{h}^*$, we are able to define the lowest weight modules we are interested in.
Fixing $\{\epsilon_1,\dots,\epsilon_n\}$ as basis for~$\mathfrak{h}^*$, the notation~$(p_1,\dots,p_n)$ will from now on be used for the weight $\sum_{i=1}^n p_i\epsilon_i\in\mathfrak{h}^*$.

\begin{Definition}
Given $p\in\mathbb{N}$, let $\mathcal{H}_n(p)$ be the unitary irreducible lowest weight module of~$\mathfrak{osp}(1|2n)$ of lowest weight $\left(\frac{p}{2},\dots,\frac{p}{2}\right)$ and with lowest weight vector $|0\rangle$.
The actions of the generators $B^+_i$ and~$B^-_i$, for~$i\in\{1,\dots,n\}$, of~$\mathfrak{osp}(1|2n)$ on~$\mathcal{H}_n(p)$ are defined by relations
\begin{gather}
\label{sec2_eq_lw-action-on-vacuum}
B^-_i|0\rangle=0,
	\qquad\text{ and }\qquad
	\big\{B^+_i,B^-_j\big\}|0\rangle=p\delta_{i,j}|0\rangle,
	\end{gather}
for all $i,j\in \{1,\dots,n\}$.
	\end{Definition}
These modules were originally introduced in the context of parastatistics with $\mathcal{H}_n(p)$ as the Fock spaces of~$n$ parabosonic particles of order $p$. Here they are usually referred to as paraboson Fock spaces \cite{Green-1953, Greenberg-Messaih-1965}.
When $p=1$, the parabosonic particles revert to usual bosonic particles, and~$\mathcal{H}_n(1)$ becomes the usual boson Fock space.
From a different point of view $\mathcal{H}_n(1)$ is the Hilbert space of the quantum harmonic oscillator and is for that reason also referred to as the oscillator representation \cite{Nishiyama-1990}.

From now on we will consider $n$ and~$p$ to be fixed positive integers.
The treatment of the representation~$\mathcal{H}_n(p)$ will be carried out through a certain polynomial realization, the construction of which will take up the rest of this section.

The Clifford algebra $\mathcal{C}\ell_p$ is generated by $p$ elements $e_\alpha$, for~$\alpha\in\{1,\dots,p\}$, satisfying
\begin{gather*}
\{e_\alpha,e_\beta\}=2\delta_{\alpha,\beta},
	\end{gather*}
for all $\alpha,\beta\in\{1,\dots,p\}$.
Let $\mathbb{C}[\mathbb{R}^{np}]$ denote the space of polynomials in $np$ variables $x_{i,\alpha}$, for~$i\in\{1,\dots,n\}$ and~$\alpha\in\{1,\dots,p\}$, with complex coefficients, and define
\begin{gather*}
\mathcal{A}:=\mathbb{C}[\mathbb{R}^{np}]\otimes \mathcal{C}\ell_p
	\end{gather*}
to be its Clifford algebra valued counterpart. We~will in general suppress the tensor products when dealing with elements of~$\mathcal{A}$.
For each non-negative integer $n\times p$ matrix $\gamma\in M_{n,p}(\mathbb{N}_0)$ and each $\eta\in \mathbb{N}_0^p$ we define the following monomials
\begin{gather}
\label{sec2_eq_fundamental-monomials}
x^\gamma:=\prod_{i=1}^n\prod_{\alpha=1}^p x_{i,\alpha}^{\gamma_{i,\alpha}}\in \mathbb{C}[\mathbb{R}^{np}] \qquad\text{and}\qquad e^{\eta}:=e_1^{\eta_1}\cdots e_p^{\eta_p}\in \mathcal{C}\ell_p.
	\end{gather}
In the context of such monomials we will refer to $\gamma$ as the exponent matrix. We~note that $e^\eta= e^{\eta'}$ if and only if $\eta_\alpha \equiv \eta_\alpha'\mod 2$, for all $\alpha\in\{1,\dots,p\}$. As~a vector space $\mathcal{A}$ has a basis consisting of the vectors $x^\gamma e^\eta$, for~$\gamma\in M_{n,p}(\mathbb{N}_0)$ and~$\eta\in \mathbb{Z}_2^p$. The basis is orthonormal with respect to the following canonical inner product
\begin{gather}
\label{sec2_eq_inner-product}
\langle x^\gamma e^\eta,x^{\gamma'} e^{\eta'}\rangle := \delta_{\gamma,\gamma'}\delta_{\eta,\eta'},
	\end{gather}
for all $\gamma,\gamma'\in M_{n,p}(\mathbb{N}_0)$ and~$\eta,\eta'\in \mathbb{Z}_2^p$.

Denoting the $2np$ odd generators of~$\mathfrak{osp}(1|2n)n{np}$ by $B^{\pm}_{i,\alpha}$ we can consider $\mathbb{C}[\mathbb{R}^{np}]$ as an $\mathfrak{osp}(1|2n)n{np}$-module with action
\begin{gather*}
B^+_{i,\alpha}x^\gamma:= x_{i,\alpha}x^\gamma \qquad\text{and}\qquad B^-_{i,\alpha}x^\gamma := \partial_{i,\alpha}x^\gamma,
	\end{gather*}
for all $i\in\{1,\dots,n\}$, $\alpha\in\{1,\dots,p\}$ and~$\gamma\in M_{n,p}(\mathbb{N}_0)$, where $\partial_{i,\alpha}:=\frac{\partial}{\partial x_{i,\alpha}}$. The relations
\begin{gather*}
[\partial_{i,\alpha},x_{j,\beta}]=\delta_{i,j}\delta_{\alpha,\beta},
	\end{gather*}
for~$i,j\in\{1,\dots,n\}$ and~$\alpha,\beta\in\{1,\dots,p\}$, imply the structural relations~\eqref{sec2_eq_PCR} of~$\mathfrak{osp}(1|2n)n{np}$. In~fact $\mathbb{C}[\mathbb{R}^{np}]\cong \mathcal{H}_{np}(1)$ as $\mathfrak{osp}(1|2n)n{np}$ modules.
Similarly we can consider $\mathcal{A}$ as an $\mathfrak{osp}(1|2n)$-module by defining operators
\begin{gather}
\label{sec2_eq_operator-realization}
X_i:=\sum_{\alpha=1}^p x_{i,\alpha}e_\alpha
	\qquad\text{and}\qquad
D_i:=\sum_{\alpha=1}^p \partial_{i,\alpha} e_\alpha,
	\end{gather}
and letting
\begin{gather}
B_i^+(x^\gamma e^\eta):= X_i(x^\gamma e^\eta)
= \sum_{\alpha=1}^p x_{i,\alpha}x^\gamma e_\alpha e^\eta,
\\
B_i^-(x^\gamma e^\eta):= D_i(x^\gamma e^\eta)
= \sum_{\alpha=1}^p \partial_{i,\alpha}x^\gamma e_\alpha e^\eta,
\label{sec2_eq_action-X_i-D_i}
\end{gather}
for all $i\in\{1,\dots,n\}$, $\gamma\in M_{n,p}(\mathbb{N}_0)$ and~$\eta\in\mathbb{N}_0^p$.
It is easily verified that $\mathcal{A}$ is an $\mathfrak{osp}(1|2n)$-module by checking that the actions~\eqref{sec2_eq_action-X_i-D_i} satisfy the relations~\eqref{sec2_eq_PCR}. It can furthermore easily be verified that the inner product~\eqref{sec2_eq_inner-product} on~$\mathcal{A}$ respects the unitary structure of~$\mathfrak{osp}(1|2n)$, thus making $\mathcal{A}$ a unitary $\mathfrak{osp}(1|2n)$-module.

The subsequent result, Proposition~\ref{sec2_prop_irrep}, tells us that $\mathcal{A}$ has an irreducible submodule isomorphic to $\mathcal{H}_n(p)$.
The shape of the operators in~\eqref{sec2_eq_operator-realization} is that of the Green's ansatz \cite{Green-1953} with the internal components here being realized using Clifford algebra elements. This type of realization of the Green's ansatz dates back to \cite{Greenberg-Macrae-1983, Greenberg-Messaih-1965}.

The module $\mathcal{A}$ admits the Howe dual pair \cite{Cheng-Zhang-2004} $(\mathfrak{osp}(1|2n),G)$, where $G= \mathop{\rm Pin}(p)$ when $p$ is even and~$G=\mathop{\rm Spin}(p)$ when $p$ is odd. This gives a multiplicity free decomposition of~$\mathcal{A}$ into a~direct sum of irreducible modules of~$\mathfrak{osp}(1|2n)\times G$ of the form $\mathcal{H}_{\rm osp}(\mu)\otimes \mathcal{H}_G(\nu)$ \cite{Cheng-Kwon-Wang-2010,Cheng-Wang-2012, Salom-2013}. Here~$\mathcal{H}_{\rm osp}(\mu)$ denotes an irreducible lowest weight module of~$\mathfrak{osp}(1|2n)$ of lowest weight $\mu$, and~$\mathcal{H}_G(\nu)$ denotes an irreducible highest weight module of~$G$ of highest weight $\nu$.

We are interested in exactly one of the irreducible modules in this decomposition, namely the one with the constant polynomial $1\otimes I\in\mathcal{A}$ as lowest weight vector. Here $I\in\mathcal{C}\ell_p$ is the identity element in the Clifford algebra. We~shall from now on use the notation~$1:=1\otimes I\in \mathcal{A}$.

\begin{Definition}
Let $L_n(p)$ be the submodule of~$\mathcal{A}$ generated by the action of~$\mathfrak{osp}(1|2n)$ on~$1\in \mathcal{A}$.
	\end{Definition}

\begin{Proposition}\label{sec2_prop_irrep}
The module $L_n(p)$ is equivalent to the irreducible lowest weight $\mathfrak{osp}(1|2n)$-module $\mathcal{H}_n(p)$.
\end{Proposition}
\begin{proof}
In the module $\mathcal{A}$ the basis for the Cartan subalgebra of~$\mathfrak{osp}(1|2n)$, as described in~\eqref{sec2_eq_cartan}, is represented by operators
\begin{gather*}
h_i\mapsto \frac{1}{2}\{X_i,D_i\} =\frac{p}{2}+\sum_{\alpha=1}^p x_{i,\alpha}\partial_{i,\alpha},
	\end{gather*}
for all $i\in\{1,\dots,p\}$. Since
\begin{gather*}
D_i(1)=0\qquad\text{and}\qquad \frac{1}{2}\{X_i,D_j\}(1)=\frac{p}{2}\delta_{i,j}(1),
	\end{gather*}
for all $i,j\in \{1,\dots,n\}$, it follows that $1$ is a lowest weight vector of weight
$\left(\frac{p}{2},\dots,\frac{p}{2}\right)$. As~mentioned above, the module $\mathcal{A}$ decomposes into a direct sum of irreducible lowest weight modules of~$\mathfrak{osp}(1|2n)$ as a result of admitting the Howe dual pair $(\mathfrak{osp}(1|2n),G)$. It follows that $L_n(p)$ must be one of the components in the decomposition and hence be irreducible.
It can thus be concluded that $L_n(p)$ is equivalent to $\mathcal{H}_n(p)$.
\end{proof}
	
In the following sections we will construct a basis for~$L_n(p)$ together with formulas for the action of~$\mathfrak{osp}(1|2n)$ on the basis.

\section[The induced module Vn(p)]{The induced module $\boldsymbol{\overline{V}_n(p)}$}\label{sec3}
In this section we introduce the induced module $\overline{V}_n(p)$ which was used in \cite{Lievens-Stoilova-VanderJeugt-2008} to study the paraboson Fock space. We~relate this module to $L_n(p)$ via a surjective module homomorphism $\Psi_p$ and prove that the kernel of this map can be identified with the maximal non-trivial submodule of~$\overline{V}_n(p)$. The relationship between $\overline{V}_n(p)$ and~$L_n(p)$ described by $\Psi_p$ will be used throughout the following chapters and serves as a valuable tool in the proof of Theorem~\ref{sec5_theo_basis}.

Following \cite{Lievens-Stoilova-VanderJeugt-2008} we consider the parabolic subalgebra of~$\mathfrak{osp}(1|2n)$ given by
\begin{gather*}
\mathfrak{P}=\operatorname{span}\big\{\big\{B_i^+,B_j^-\big\},B_i^-,\{B_i^-,B_j^-\}\colon i,j\in\{1,\dots,n\}\big\}.
	\end{gather*}
Let $|0\rangle$ describe the lowest weight vector of the module $\mathcal{H}_n(p)$, then the action~\eqref{sec2_eq_lw-action-on-vacuum} of~$\mathfrak{P}$ on~$|0\rangle$ generates a one dimensional $\mathfrak{P}$-module denoted $\mathbb{C}|0\rangle$. As~in \cite{Lievens-Stoilova-VanderJeugt-2008} we let $\overline{V}_n(p)$ denote the induced module of~$\mathfrak{osp}(1|2n)$ relative to $\mathfrak{P}$, see \cite{Frappat-Sorba-Sciarrino-2000},
\begin{gather*}
\overline{V}_n(p):= \operatorname{Ind}_\mathfrak{P}^{\mathfrak{osp}(1|2n)} \mathbb{C}|0\rangle.
	\end{gather*}
We let $V_n(p)$ denote the irreducible module obtained by taking the quotient of~$\overline{V}_n(p)$ by its maximal non-trivial submodule $M_n(p)$,
\begin{gather*}
V_n(p):=\overline{V}_n(p)/M_n(p).
	\end{gather*}
By construction~$V_n(p)$ is then isomorphic to $\mathcal{H}_n(p)$ as an $\mathfrak{osp}(1|2n)$-module.

Let $\mathfrak{osp}(1|2n)^+$ denote the subalgebra of~$\mathfrak{osp}(1|2n)$ generated by $B_i^+$, for~$i\in\{1,\dots,n\}$,
\begin{gather*}
\mathfrak{osp}(1|2n)^+:=\operatorname{span}\big\{B_i^+,\big\{B_i^+,B_j^+\big\}\colon i,j\in\{1,\dots,n\}\big\}.
	\end{gather*}
As a vector space $\mathfrak{osp}(1|2n)=\mathfrak{osp}(1|2n)^+\oplus\mathfrak{P}$. The definition of the induced module $\overline{V}_n(p)$ then implies the following result that will be used in the proof of Theorem~\ref{sec5_theo_basis}.
\begin{Lemma}
\label{sec3_lemm_basis-positive-enveloping-algebra}
The map
\begin{gather*}
\Phi_p\colon \ U(\mathfrak{osp}(1|2n)^+)\rightarrow \overline{V}_n(p),\qquad
B\mapsto B|0\rangle,
	\end{gather*}
is an isomorphism of vector spaces.
	\end{Lemma}

We relate the modules $\overline{V}_n(p)$ and~$L_n(p)$ by the module homomorphism
\begin{gather*}
\Psi_p\colon\ \overline{V}_n(p)\to L_n(p),
	\end{gather*}
for which $\Psi_p(|0\rangle):=1$. This means that $\Psi_p(B_i^+v)=X_i\Psi_p(v)$, and~$\Psi_p(B_i^-v)=D_i\Psi_p(v)$, for all $i\in\{1,\dots,n\}$ and~$v\in\overline{V}_n(p)$.
To continue, the following notation is needed.
For each $k\in\mathbb{N}$, consider the index set
\begin{gather*}
\mathcal{I}(k):= \big\lbrace (i_1,\dots,i_k)\in \{1,\dots,n\}^k\colon i_1<\cdots < i_k \big\rbrace.
	\end{gather*}
To each $I=(i_1,\dots,i_k)\in \mathcal{I}(k)$ we assign the following antisymmetric sums of operators acting on~$L_n(p)$ and~$\overline{V}_n(p)$ respectively.
\begin{gather}
\label{sec3_eq_polynomial-column-operator}
X_I:= \sum_{\sigma\in S_k} \operatorname{sgn}(\sigma)X_{i_{\sigma(1)}}\cdots X_{i_{\sigma(k)}}
	\end{gather}
and
\begin{gather}
\label{sec3_eq_column-operator}
B_I^+:=\sum_{\sigma\in S_k} \operatorname{sgn}(\sigma)B^+_{i_{\sigma(1)}}\cdots B^+_{i_{\sigma(k)}},
	\end{gather}
where $S_k$ denotes the $k$'th symmetric group. Whereas $B_I^+\neq 0$, for all $I\in \mathcal{I}(k)$ and~$k\in\mathbb{N}$, it~holds that $X_I=0$ if $k>p$ and~$X_I\neq 0$ if $k\leq p$. This is verified by a short calculation showing that
\begin{gather}
\label{sec3_eq_poly-column-operator-expansion}
X_I= k!\sum_{
\begin{subarray}{c}\alpha_1,\dots,\alpha_k\in \{1,\dots,p\},\\ \alpha_i\neq \alpha_j, \text{ when } i\neq j\end{subarray}}x_{i_1,\alpha_{1}}\cdots x_{i_k,\alpha_{k}}e_{\alpha_1}\cdots e_{\alpha_k},
	\end{gather}
where the sum is over all sets of~$k$ mutually distinct integers in $\{1,\dots,p\}$.

\begin{Lemma}
\label{sec3_lemm_kernel-equals-maximal-submodule}
The map $\Psi_p$ is a surjective $\mathfrak{osp}(1|2n)$-module homomorphism with kernel
\begin{gather*}
\ker \Psi_p = M_n(p)
	\end{gather*}
generated by the action of~$\mathfrak{osp}(1|2n)$ on the vectors $B_I^+|0\rangle$ with $I\in\mathcal{I}(p+1)$.
	\end{Lemma}
\begin{proof}
The surjectivity of the $\mathfrak{osp}(1|2n)$-module homomorphism $\Psi_p$ follows essentially from the way it is constructed.
The discussion preceding~\eqref{sec3_eq_poly-column-operator-expansion} implies that $B_I^+|0\rangle\in \ker \Psi_p$, for all $I\in\mathcal{I}(p+1)$, and that these are non-zero vectors.
The maximality of~$M_n(p)$ implies that $\ker \Psi_p \subset M_n(p)$, and thus that $B_I^+|0\rangle\in M_n(p)$, for all $I\in\mathcal{I}(p+1)$.

From the construction in \cite{Lievens-Stoilova-VanderJeugt-2008} of a Gelfand--Zetlin basis for~$\overline{V}_n(p)$ it follows that $M_n(p)$ is generated by the action of~$\mathfrak{osp}(1|2n)$ on the Gelfand--Zetlin basis vectors with top row
\begin{gather*}
(\underbrace{1,1,\dots,1}_{p+1},0,0,\dots,0),
	\end{gather*}
in their Gelfand--Zetlin labelling. These vectors can be naturally indexed by the set $\mathcal{I}(p+1)$ such that the vector with index $I\in\mathcal{I}(p+1)$ has weight $\epsilon_I:=\frac{p}{2}\sum_{i=1}^n\epsilon_i+\sum_{k=1}^{p+1}\epsilon_{i_k}$.
The weight space of~$M_n(p)$ corresponding to the weight $\epsilon_I$ is 1-dimensional and~$B_I^+|0\rangle$ is a vector of weight $\epsilon_I$.
This means that $M_n(p)$ is generated by the action of~$\mathfrak{osp}(1|2n)$ on the vectors $B_I^+|0\rangle$, for~$I\in\mathcal{I}(p+1)$, and thus that $\ker \Psi_p = M_n(p)$.
\end{proof}

Recalling that $\mathcal{H}_n(p)\cong V_n(p)=\overline{V}_n(p)/M_n(p)$ this result implies Proposition~\ref{sec2_prop_irrep}.
Furthermore it implies that $\overline{V}_n(p)\cong L_n(p)$ when $p\geq n$, since in that case $\mathcal{I}(p+1)=\varnothing$ and thus $M_n(p)=\{0\}$.
This observation is in fact closely related to a more general result~\cite{Colombo-Sommen-Sabadini-Struppa-2004}, \cite[Theorem 1.1]{Lavicka-Soucek-2017}. In~the context of the present paper the result can be stated as follows.
Let $\mathcal{W}$ be any irreducible factor in the decomposition of~$\mathcal{A}$. Denote the lowest weight of~$\mathcal{W}$ by $\mu\in \mathfrak{h}^*$ and let $\mathcal{V}(\mu)$ denote the corresponding induced module, constructed in much the same way as we constructed $\overline{V}_n(p)$. Then $\mathcal{W}\cong \mathcal{V}(\mu)$ whenever $p\geq 2n$. As~mentioned above this condition reduces to $p\geq n$ for the modules we study here, namely the paraboson Fock spaces. This is a consequence of the fact that they have a unique vacuum, contrary to a general irreducible factor of~$\mathcal{A}$.

\section[Tableau vectors in Ln(p) and Vn(p)]
{Tableau vectors in $\boldsymbol{L_n(p)}$ and~$\boldsymbol{\overline{V}_n(p)}$}\label{sec4}

The main result of this paper is the explicit construction of bases for the modules $L_n(p)$ and~$\overline{V}_n(p)$ which are related by the map $\Psi_p$. In~this section we will construct the sets of vectors that will be our candidates for the desired bases, see Definition~\ref{sec4_defi_basis-elements}. Proving that they are indeed bases for the modules is postponed to the next section. As~touched upon in Section~\ref{sec3}, bases for the modules $V_n(p)$ and~$\overline{V}_n(p)$, which are related by the quotient map
\begin{gather*}
\overline{V}_n(p)\to \overline{V}_n(p)/M_n(p)=V_n(p),
	\end{gather*}
are found and studied in the paper \cite{Lievens-Stoilova-VanderJeugt-2008}.
These basis vectors are identified only with their labellings by Gelfand--Zetlin patterns relative to the $\mathfrak{gl}(n)$ subalgebra of~$\mathfrak{osp}(1|2n)$. The new basis constructed in this section, on the contrary, will consist of explicit polynomials in $L_n(p)$.

Partitions, Young tableaux and symmetric functions will be used throughout this paper. In~the interest of consistency, we will be using the notation established in Macdonald \cite{Macdonald-2015}.
Let~the set of all partitions be denoted by $\mathcal{P}$, and let $\ell(\lambda)$ be the length of the partition~$\lambda\in \mathcal{P}$.
Let~$\mathbb{Y}_n$ denote the set of semistandard (s.s.) Young tableaux with weights in $\mathbb{N}_0^n$, that is with numbers in the set $\{1,\dots,n\}$ as entries. Let $\lambda_A\in\mathcal{P}$ and~$\mu_A\in\mathbb{N}_0^n$ denote the shape and weight of~$A\in\mathbb{Y}_n$ respectively and let
\begin{gather*}
\mathbb{Y}_n(p)=\{A\in \mathbb{Y}_n\colon \ell(\lambda_A)\leq p\}
	\end{gather*}
be the set of s.s.\ Young tableaux in $\mathbb{Y}_n$ with at most $p$ rows. S.s.\ Young tableaux are enumerated by Kostka numbers
\begin{gather}
\label{sec4_eq_kostka}
K_{\lambda,\mu}:=\#\{\text{s.s.\ Young tableaux of shape }\lambda\text{ and weight }\mu\}.
	\end{gather}

The first use we shall make of this terminology is to describe the weight space dimensions of~$L_n(p)$ and~$\overline{V}_n(p)$ as sums of Kostka numbers. This description will serve as the main inspiration for defining basis vectors for the modules.
The action of~$B_i^\pm$ on any weight vector of~$L_n(p)$ or $\overline{V}_n(p)$ changes its weight by $\pm \epsilon_i$, for any $i\in\{1,\dots,n\}$. Therefore, any weight corresponding to a non-zero weight space of either module will be of the form
\begin{gather*}
\mu+\frac{p}{2} := \mu+\bigg(\frac{p}{2},\dots,\frac{p}{2}\bigg)\in \mathfrak{h}^*,
	\end{gather*}
for some $\mu\in\mathbb{N}_0^n$. We~will use the notations $\overline{V}_n(p)_{\mu+\frac{p}{2}}$ and~$L_n(p)_{\mu+\frac{p}{2}}$ respectively for the weight spaces of~$\overline{V}_n(p)$ and~$L_n(p)$ corresponding to the weight $\mu+\frac{p}{2}$, for any $\mu\in\mathbb{N}_0^n$.

\begin{Lemma}
\label{sec4_lemm_dim-weight-spaces}
For any $\mu\in \mathbb{N}_0^n$,
\begin{gather}
\label{sec4_eq_dim-weight-spaces-irrep}
\dim\overline{V}_n(p)_{\mu+\frac{p}{2}} = \sum_{\lambda \in \mathcal{P}} K_{\lambda, \mu},
	\qquad\text{and}\qquad
	\dim L_n(p)_{\mu+\frac{p}{2}} = \sum_{\begin{subarray}{c}\lambda \in \mathcal{P}, \\ \ell(\lambda)\leq p\end{subarray}} K_{\lambda, \mu}.
	\end{gather}
	\end{Lemma}
\begin{proof}
The character formulas of~$\overline{V}_n(p)$ and~$V_n(p)$ were obtained in the papers \cite{Loday-Jean-Louis-Popov-2008}, \cite[Theo\-rem~7]{Lievens-Stoilova-VanderJeugt-2008}. Recalling that $V_n(p)\cong L_n(p)$ these character formulas are
\begin{gather*}
\operatorname{char} \overline{V}_n(p) = (t_1\cdots t_n)^{\frac{p}{2}}\sum_{\lambda \in \mathcal{P}} s_\lambda(t_1,\dots, t_n),
	\end{gather*}
and
\begin{gather*}
\operatorname{char} L_n(p) = (t_1\cdots t_n)^{\frac{p}{2}}\sum_{\begin{subarray}{c}\lambda \in \mathcal{P}, \\ \ell(\lambda)\leq p\end{subarray}} s_\lambda(t_1,\dots, t_n),
	\end{gather*}
where $s_\lambda$ denotes the Schur function indexed by $\lambda\in\mathcal{P}$, and~$t_i$ denotes the formal exponential $ e^{\epsilon_i}$.
The lemma then follows by recalling two facts.
First, in the character formulas the coefficients of monomials $t_1^{\mu_1+\frac{p}{2}}\cdots t_n^{\mu_n+\frac{p}{2}}$, for~$\mu\in\mathbb{N}_0^n$, describe the dimensions of the weight spaces of the relevant modules corresponding to the weight $\mu+\frac{p}{2}$. Second, by \cite{Stanley-1999}, the monomial expansion of the Schur function indexed by $\lambda\in\mathcal{P}$ is
\begin{gather*}
s_\lambda(t_1,\dots, t_n) = \sum_{\mu\in\mathbb{N}_0^n} K_{\lambda,\mu}t_1^{\mu_1}\cdots t_n^{\mu_n}.
\tag*{\qed}
\end{gather*}
\renewcommand{\qed}{}
\end{proof}

Recalling how the Kostka numbers enumerate s.s.\ Young tableaux we may interpret Lem\-ma~\ref{sec4_lemm_dim-weight-spaces} as saying that the modules $\overline{V}_n(p)$ and~$L_n(p)$ admit bases indexed by the tableaux in $\mathbb{Y}_n$ and~$\mathbb{Y}_n(p)$ respectively in such a way that the weights of the s.s.\ Young tableaux are related to the weights of their corresponding basis vectors by the addition of~$\big(\frac{p}{2},\ldots,\frac{p}{2}\big)$.

The existence of such bases is not significant in itself unless we can find relatively simple descriptions of the vectors in terms of the corresponding tableaux.
For the construction of the basis vectors for~$L_n(p)$ and~$\overline{V}_n(p)$ the idea is to interpret each column of a given s.s.\ Young tableau $A\in \mathbb{Y}_n$ as an operator of the form~\eqref{sec3_eq_polynomial-column-operator} or~\eqref{sec3_eq_column-operator} and then let each such operator act on~the lowest weight vector with the leftmost column acting first.

To formalize this we need some notation.
Given a partition~$\lambda\in \mathcal{P}$, we denote the conjugate partition by $\lambda'$. Let $a_l$ denote the $l$'th column, counted from left to right, of~$A\in\mathbb{Y}_n$, for all $l\in\{1,\dots,\ell(\lambda_A')\}$.
Since the entries of the columns of~$A$ are strictly increasing downwards we can naturally consider $a_l$ as an element of~$\mathcal{I}((\lambda_A')_l)$, for all $l\in\{1,\dots,\ell(\lambda_A')\}$, noting here that the amount of columns in $A$ is indeed $\ell(\lambda_A')$.

\begin{Definition}
\label{sec4_defi_basis-elements}
Given $A\in\mathbb{Y}_n$ we let $m=\ell(\lambda_A')$ and define
\begin{gather*}
v_A(p):= B_{a_m}^+\cdots B_{a_1}^+|0\rangle\in \overline{V}_n(p),
	\qquad\text{and}\qquad
	\omega_A(p):= X_{a_m}\cdots X_{a_1}(1)\in L_n(p).
	\end{gather*}
	\end{Definition}
In Section~\ref{sec5} we will prove that these vectors form bases for~$\overline{V}_n(p)$ and~$L_n(p)$ respectively. There it will be useful to know that they satisfy the following properties.
\begin{Lemma}
\label{sec4_lemm_basis-vector-properties}
Let $A\in \mathbb{Y}_n$, then $v_A(p)$ and~$\omega_A(p)$ both have weight $\mu_A+\frac{p}{2}$. Furthermore
\begin{gather*}
\Psi_p(v_A(p))=
\begin{cases}
\omega_A(p), & \text{ if }\ A\in\mathbb{Y}_n(p),\\
0, & \text{ if }\ A\notin\mathbb{Y}_n(p).
\end{cases}
	\end{gather*}
	\end{Lemma}	
	
In order to illustrate the above-mentioned concepts, consider an example. Let $n=5$, $\lambda_A=(4,3,1)$, $\mu_A=(2,2,3,1,0)$ and
\begin{gather*}
A=
\ytableausetup{centertableaux,boxsize=1.1em}
\ytableaushort{1123,233,4}
	\end{gather*}
The columns of~$A$ are then $a_1=(1,2,4)$, $a_2=(1,3)$, $a_3=(2,3)$ and~$a_4=(3)$. In~that case
\begin{gather*}
v_A(p)=B^+_{3}B^+_{(2,3)}B^+_{(1,3)}B^+_{(1,2,4)}|0\rangle\qquad\text{and}\qquad \omega_A(p)=X_{3}X_{(2,3)}X_{(1,3)}X_{(1,2,4)}(1),
	\end{gather*}
where $B_{a_l}^+$ and~$X_{a_l}$, for~$l\in\{1,\dots,4\}$, are defined in~\eqref{sec3_eq_column-operator} and~\eqref{sec3_eq_polynomial-column-operator} respectively. In~the case $p \geq 3$, $v_A(p)\neq 0$, $\omega_A(p)\neq 0$. On the other hand, if $p<3$, then $v_A(p)\neq 0$ and~$\omega_A(p)=0$.

\section[Bases for Ln(p) and Vn(p)]
{Bases for~$\boldsymbol{L_n(p)}$ and~$\boldsymbol{\overline{V}_n(p)}$}\label{sec5}

The goal of this section is to prove that the sets $\{v_A(p)\colon A\in \mathbb{Y}_n\}$ and~$\{\omega_A(p)\colon A\in \mathbb{Y}_n(p)\}$, defined in Definition~\ref{sec4_defi_basis-elements}, form bases for~$\overline{V}_n(p)$ and~$L_n(p)$. This is the content of Theorem~\ref{sec5_theo_basis}.
The main difficulty lies in proving mutual linear independence of the vectors in $\{\omega_A(p)\colon A\in \mathbb{Y}_n(p)\}$, the rest is achieved by applying results from the previous sections.
To prove this linear independence we construct a total ordering of~$\mathbb{Y}_n(p)$ with the following property
\begin{gather*}
\omega_A(p) \notin \operatorname{span}\{\omega_B(p)\colon B\in \mathbb{Y}_n(p), B<A\}.
	\end{gather*}
Given that the weight spaces of~$L_n(p)$ are finite dimensional by Lemma~\ref{sec4_lemm_dim-weight-spaces}, the existence of~such a total ordering implies linear independence.

\begin{Definition}
Given $A\in\mathbb{Y}_n$ and~$k\in \{1,\dots,n\}$ the $k$'th subtableau of~$A$ is defined to be the tableau $A^k\in\mathbb{Y}_n$ obtained by truncating $A$ to only the entries containing the numbers $1,\dots,k$.
	\end{Definition}

For example, we have
\begin{gather*}
A=
\ytableausetup{centertableaux,boxsize=1.1em}
\ytableaushort{1124,234}
\implies
A^4=A,
\quad\
A^3=
\ytableausetup{centertableaux,boxsize=1.1em}
\ytableaushort{112,23}\,,
\quad\
A^2=
\ytableausetup{centertableaux,boxsize=1.1em}
\ytableaushort{112,2}
\quad\ \text{and}\quad\
A^1=
\ytableausetup{centertableaux,boxsize=1.1em}
\ytableaushort{11}\,.
	\end{gather*}

To define the total ordering of~$\mathbb{Y}_n$ we endow both $\mathbb{N}_0^n$ and~$\mathcal{P}$ with the graded lexicographic ordering, both denoted by $<$. The definition of these orderings are given in the Appendix~\ref{app1}.

\begin{Definition}
\label{sec5_defi_ordering}
Given $A,B\in \mathbb{Y}_n$ we write $A<B$ if $\mu_A<\mu_B$ in $\mathbb{N}_0^n$, or if $\mu_A=\mu_B$ and there exists $k\in\{1,\dots,n\}$ such that, for all $l<k$,
\begin{gather*}
\lambda_{A^{l}}=\lambda_{B^{l}}\qquad\text{and}\qquad \lambda_{A^k}<\lambda_{B^k}\text{ in } \mathcal{P}.
	\end{gather*}
	\end{Definition}
The relation~$<$ on~$\mathbb{Y}_n$ defined above is a total order. This follows from the fact that the graded lexicographic order gives a total ordering of both $\mathbb{N}_0^n$ and~$\mathcal{P}$.
The total ordering of~$\mathbb{Y}_n$ defined above is inherited by $\mathbb{Y}_n(p)$.

To get a feel for this ordering of~$\mathbb{Y}_n$ consider the following example. Let $n=5$, and consider the ascending chain of all 13 s.s.\ Young tableaux of weight $\mu=(2,1,1,1,0)$:
\begin{gather*}
\ytableausetup{centertableaux,boxsize=1.1em,}
\ytableaushort{11,2,3,4}\, <\ \ytableaushort{11,24,3}\ <\
\ytableaushort{114,2,3}\, <\ \ytableaushort{11,23,4}\ <\
\ytableaushort{114,23}\, <\ \ytableaushort{113,2,4}\ <\
\ytableaushort{113,24}\ <\ \ytableaushort{1134,2}
\\ \hphantom{\ytableaushort{11,2,3,4}\ }
\ytableausetup{centertableaux,boxsize=1.1em,}
{}<\ \ytableaushort{112,3,4}\ <\
\ytableaushort{112,34}\ <\ \ytableaushort{1124,3}\ <\
\ytableaushort{1123,4}\ <\ \ytableaushort{11234}\,.
	\end{gather*}

To see the connection between this ordering and the prospective basis for~$L_n(p)$ we need to consider the expansion of the vectors $\omega_A(p)$ into monomials. To do so, a new class of tableaux is needed.
Let
\begin{gather*}
\mathbb{T}(\lambda,p) :=
\begin{Bmatrix}
\text{ Fillings of the Young diagram of shape $\lambda\in \mathcal{P}$ with}\\
\text{ numbers $1,\dots,p$ occuring at most once in each column}
	\end{Bmatrix}.
	\end{gather*}
We will refer to the tableaux in $\mathbb{T}(\lambda,p)$ as column distinct (c.d.) Young tableaux.
Some examples of these are:
\begin{gather*}
\ytableausetup{centertableaux,boxsize=1.1em}
\ytableaushort{2121,3333,14}\,,
\quad
\ytableaushort{3121,2433,13}\,,
\quad
\ytableaushort{3133,2421,13}
\quad\in
\mathbb{T}((4,4,2),4).
	\end{gather*}

From now on we shall adopt the slightly abusive notation of writing $\lambda$ both for a partition in~$\mathcal{P}$ and for the set of coordinates in the corresponding Young diagram. That is,
\begin{gather*}
\lambda=\big\{(k,l)\colon l\in \{1,\dots,\ell(\lambda')\}\ \text{and}\ k\in \{1,\dots, \lambda_l'\}\big\},
	\end{gather*}
for all $\lambda\in \mathcal{P}$.
With the notion of c.d.\ Young tableaux the following monomials in $\mathbb{C}[\mathbb{R}^{np}]$ and~$\mathcal{C}\ell_p$ can be defined.
Given a tableau $C\in \mathbb{T}(\lambda,p)$ and~$A\in \mathbb{Y}_n(p)$ of shape $\lambda=\lambda_A\in \mathcal{P}$. We~denote the entries of~$A$ and~$C$ by $a_{k,l}$ and~$c_{k,l}$ respectively, for all $(k,l)\in \lambda$. Letting $m=\ell(\lambda')$ be the number of columns in $A$ and~$C$ we define monomials
\begin{gather}
\label{sec5_eq_clifford-monomial-def}
e_C:=(e_{c_{1,m}}\cdots e_{c_{\lambda_m',m}})\cdots (e_{c_{1,1}}\cdots e_{c_{\lambda_1',1}})\in \mathcal{C}\ell_p
	\end{gather}
and
\begin{gather}
\label{sec5_eq_monomial-def}
x_{A,C}:=
\prod_{(k,l)\in \lambda} x_{a_{k,l},c_{k,l}}\in \mathbb{C}[\mathbb{R}^{np}].
	\end{gather}	
For example, consider the case $n=4$, $p=4$, and
\begin{gather*}
A=\ytableausetup{centertableaux,boxsize=1.1em}
\ytableaushort{12,23,34}
\qquad \text{and} \qquad
C=\ytableaushort{23,31,14}\, .
	\end{gather*}
The monomials defined in~\eqref{sec5_eq_clifford-monomial-def} and~\eqref{sec5_eq_monomial-def} will then take the form
\begin{gather*}
e_C= e_3e_1e_4e_2e_3e_1= e_2e_4
	\end{gather*}
and
\begin{gather*}
x_{A,C}=x_{1,2}x_{2,3}^2x_{3,1}^2x_{4,4}.
	\end{gather*}
To each partition~$\lambda\in \mathcal{P}$ we associate the following factorial:
\begin{gather*}
\lambda!:= \lambda_1'!\cdots \lambda_{\lambda_1}'!\in \mathbb{N}.
	\end{gather*}

Recalling how $\omega_A(p)$, $e_C$ and~$x_{A,C}$ where defined in Definition~\ref{sec4_defi_basis-elements},~\eqref{sec5_eq_clifford-monomial-def} and~\eqref{sec5_eq_monomial-def}, respectively, a short calculation using~\eqref{sec3_eq_poly-column-operator-expansion} gives the following monomial expansion of~$\omega_A(p)$.

\begin{Lemma}
\label{sec5_lemm_monomial-expansion}
For all $A\in \mathbb{Y}_n(p)$, we have
\begin{gather*}
\omega_A(p)=\lambda_A!\sum_{C\in\mathbb{T}(\lambda_A,p)} x_{A,C}e_C.
	\end{gather*}
	\end{Lemma}

To illustrate Definition~\ref{sec4_defi_basis-elements} and Lemma~\ref{sec5_lemm_monomial-expansion} we consider as an example $n=2$, $p=2$ and
\begin{gather*}
A=\ytableausetup{centertableaux,boxsize=1.1em}\ytableaushort{112,2}
	\end{gather*}
Definition~\ref{sec4_defi_basis-elements} gives
\begin{gather*}
\omega_A(p) = X_2 X_1 X_{1,2}
	=2!(x_{2,1}e_1+x_{2,2}e_2)(x_{1,1}e_1+x_{1,2}e_2)(x_{1,1}x_{2,2}e_1e_2+x_{1,2}x_{2,1}e_2e_1).
	\end{gather*}
On the other hand Lemma~\ref{sec5_lemm_monomial-expansion} gives
\begin{gather*}
\ytableausetup{centertableaux,boxsize=0.9em}
\omega_A(p)=2x\!{}_{\left(\,\tiny\ytableaushort{112,2}\,,\,\ytableaushort{212,1}\,\right)}
e\,{}_{\tiny\ytableaushort{212,1}}
+2x\!{}_{\left(\,\tiny\ytableaushort{112,2}\,,\,\ytableaushort{121,2}\,\right)}
e\,{}_{\tiny\ytableaushort{121,2}}
+2x\!{}_{\left(\,\tiny\ytableaushort{112,2}\,,\,\ytableaushort{211,1}\,\right)}
e\,{}_{\tiny\ytableaushort{211,1}}
\\ \hphantom{\omega_A(p)=}
{}+2x\!{}_{\left(\,\tiny\ytableaushort{112,2}\,,\,\ytableaushort{222,1}\,\right)}
e\,{}_{\tiny\ytableaushort{222,1}}
+2x\!{}_{\left(\,\tiny\ytableaushort{112,2}\,,\,\ytableaushort{111,2}\,\right)}
e\,{}_{\tiny\ytableaushort{111,2}}
+2x\!{}_{\left(\,\tiny\ytableaushort{112,2}\,,\,\ytableaushort{122,2}\,\right)}
e\,{}_{\tiny\ytableaushort{122,2}}
\\ \hphantom{\omega_A(p)=}
{}+2x\!{}_{\left(\,\tiny\ytableaushort{112,2}\,,\,\ytableaushort{112,2}\,\right)}
e\,{}_{\tiny\ytableaushort{112,2}}
+2x\!{}_{\left(\,\tiny\ytableaushort{112,2}\,,\,\ytableaushort{221,1}\,\right)}
e\,{}_{\tiny\ytableaushort{221,1}}\,.
\end{gather*}
Both formulas give the expansion
\begin{gather*}
\omega_A(p)=-4x_{1,1}x_{1,2}x_{2,1}x_{2,2}-2x_{1,1}x_{1,2}x_{2,1}^2e_1e_2-2x_{1,2}^2x_{2,1}x_{2,2}e_1e_2 +2x_{1,1}^2x_{2,1}x_{2,2}e_1e_2
\\ \hphantom{\omega_A(p)=}
+2x_{1,1}x_{1,2}x_{2,2}^2e_1e_2+2x_{1,1}^2x_{2,2}^2+2x_{1,2}^2x_{2,1}^2.
	\end{gather*}

When looking at an expansion such as the one in Lemma~\ref{sec5_lemm_monomial-expansion}, it is natural to ask the following questions. Given $A\in\mathbb{Y}_n(p)$ and~$C\in\mathbb{T}(\lambda_A,p)$, does there exist $C'\in\mathbb{T}(\lambda_A,p)$ with $C\neq C'$ such that
\begin{gather*}
x_{A,C}e_C=\pm x_{A,C'}e_{C'}.
	\end{gather*}
If so, which $C'\in\mathbb{T}(\lambda_A,p)$ has this property. As~is demonstrated by the following example,~\eqref{sec5_eq_proportional-monomials}, the answer to the first question is in some instances yes, though we will shortly describe a class of c.d. Young tableaux for which the answer is always no, see~\eqref{sec5_eq_leading-tableau},
\begin{gather}
\label{sec5_eq_proportional-monomials}
\ytableausetup{centertableaux,boxsize=0.9em}
x\!{}_{\left(\,\tiny\ytableaushort{11,22}\,,\,\ytableaushort{12,21}\,\right)}
e\,{}_{\tiny\ytableaushort{12,21}}
= x_{1,2}x_{2,1}x_{1,1}x_{2,2}=
x\!{}_{\left(\,\tiny\ytableaushort{11,22}\,,\,\ytableaushort{21,12}\,\right)}
e\,{}_{\tiny
\ytableaushort{21,12}}\,.
	\end{gather}
The second question will be dealt with in more detail in Section~\ref{sec6}. However we will make some preliminary considerations here.
For any exponent matrix $\gamma\in M_{n,p}(\mathbb{N}_0)$ we denote the row and column sums as follows
\begin{gather*}
\mu_\gamma:= \left(\sum_{\alpha=1}^p\gamma_{1,\alpha},\dots,\sum_{\alpha=1}^p\gamma_{n,\alpha}\right)\in \mathbb{N}_0^n
\qquad\text{and}\qquad
\eta_\gamma:= \left(\sum_{i=1}^n\gamma_{i,1},\dots,\sum_{i=1}^n\gamma_{i,p}\right)\in \mathbb{N}_0^p.
	\end{gather*}
The definitions of~$e_C$ and~$x_{A,C}$ in~\eqref{sec5_eq_clifford-monomial-def} and~\eqref{sec5_eq_monomial-def} then imply that for each pair of tableaux $A\in \mathbb{Y}_n(p)$ and~$C\in\mathbb{T}(\lambda_A,p)$ there exists a positive integer $N(C)\in \mathbb{N}$ and a unique exponent matrix $\gamma_{A,C}\in M_{n,p}(\mathbb{N}_0)$ with $\mu_{\gamma_{A,C}}=\mu_A$ such that
\begin{gather}
\label{sec5_eq_exponent-matrix}
x_{A,C}e_C =(-1)^{N(
C)}x^{\gamma_{A,C}} e^{\eta_{\gamma_{A,C}}},
	\end{gather}
recalling that $x^\gamma$ and~$ e^{\eta_{\gamma_{A,C}}}$ were defined in~\eqref{sec2_eq_fundamental-monomials}. A~combinatorial formula for determining $(-1)^{N(
C)}$ is given in~\eqref{sec6_eq_sign-of-young-tableau}. We~note that
\begin{gather}
\label{sec5_eq_exponent-matrix-2}
(\gamma_{A,C})_{i,\alpha}= \#\{ (k,l)\in\lambda_A\colon a_{k,l}=i, c_{k,l}=\alpha\}
	\end{gather}
for all $i\in\{1,\dots,n\}$ and~$\alpha\in\{1,\dots,p\}$.
Using~\eqref{sec2_eq_inner-product}	we denote the normalized coefficient of~$x^\gamma e^{\eta_\gamma}$ in $\omega_A(p)$ by
\begin{gather*}
c_A(\gamma):=\frac{1}{\lambda_A!}\big\langle x^\gamma e^{\eta_\gamma} , \omega_A(p) \big\rangle,
	\end{gather*}
for all $\gamma\in M_{n,p}(\mathbb{N}_0)$, $\eta\in\mathbb{N}_0^p$ and~$A\in\mathbb{Y}_n(p)$.
The expansion of~$\omega_A(p)$ presented in Lemma~\ref{sec5_lemm_monomial-expansion} implies the following result.
\begin{Lemma}\label{sec5_lemm_priliminary-weight-coefficient-formula}
Let $A\in \mathbb{Y}_n(p)$, then
\begin{gather*}
\omega_A(p)=\lambda_A!\sum_{\begin{subarray}{c}\gamma\in M_{n,p}(\mathbb{N}_0) \\ \mu_\gamma=\mu_A\end{subarray}} c_A(\gamma)x^\gamma e^{\eta_\gamma},
	\end{gather*}
where
\begin{gather}
\label{sec5_eq_priliminary-weight-coefficient-formula}
c_A(\gamma)=\sum_{\begin{subarray}{c}C\in\mathbb{T}(\lambda_A,p) \\ \gamma_{A,C}=\gamma\end{subarray}} (-1)^{N(C)}.
	\end{gather}
	\end{Lemma}
Since the coefficients $c_A(\gamma)$ are integer, the expansion of~$\omega_A(p)$ into monomials $x^\gamma e^{\eta_\gamma}$ has integer coefficients. As~we shall see in Section~\ref{sec6}, this is part of the reason that the generator matrix elements turn out to be integers.

While it is true, as illustrated by~\eqref{sec5_eq_proportional-monomials}, that some terms appear multiple times, possibly with different signs, in the expansion from Lemma~\ref{sec5_lemm_monomial-expansion} other terms are more special. Specifically, it will be proven in Proposition~\ref{sec5_prop_leading-term} that the terms corresponding to c.d.\ Young tableaux of the following type are linearly independent of the rest of the terms in the expansion.
Given $A\in \mathbb{Y}_n(p)$, let
\begin{gather}
\label{sec5_eq_leading-tableau}
D_A\in \mathbb{T}(\lambda_A,p)
	\end{gather}
be the c.d.\ Young tableau with all $1$'s in the $1$st row, $2$'s in the $2$nd row and~$k$'s in the $k$th row. For these tableaux~\eqref{sec5_eq_exponent-matrix-2} implies that $\gamma_{A,D_A}$ is lower triangular and that
\begin{gather}
\label{sec5_eq_exponent-matrix-leading-term}
(\gamma_{A,D_A})_{i,\alpha}= \#\{ \text{$i$'s in the $\alpha$'th row of~$A$}\}.
	\end{gather}
As an example of these tableaux,
\begin{gather*}
\ytableausetup{centertableaux,boxsize=1.1em}
\text{if}\quad
A=
\ytableaushort{1123,2234,34}\,,\qquad
\text{then}\quad
D_A=
\ytableaushort{1111,2222,33}\,.
	\end{gather*}

\begin{Proposition}
\label{sec5_prop_leading-term}
Let $A,B\in \mathbb{Y}_n(p)$ with $A<B$ and~$\lambda=\lambda_A$, then
\begin{enumerate}\itemsep=0pt
\item[$(a)$] $c_A(\gamma_{A,D_A})=(-1)^{\sum_{j=1}^{\lambda_1}\frac{(j-1)\lambda_j'(\lambda_j'-1)}{2}}\neq 0$,
\item[$(b)$] $c_A(\gamma_{B,D_B})=0$.
	\end{enumerate}
	\end{Proposition}

In other words, the expansion of~$\omega_A(p)$ described in Lemma~\ref{sec5_lemm_monomial-expansion} contains a unique leading term $x_{A,D_A}e_{D_A}$ which is linearly independent of all the other terms in the expansion. Furthermore, if $A<B$, the leading term $x_{B,D_B}e_{D_B}$ of~$\omega_B(p)$ does not appear in the expansion of~$\omega_A(p)$.

\begin{proof}
We begin with the proof of statement $(a)$. A~short calculation based on~\eqref{sec5_eq_exponent-matrix} shows that $(-1)^{N(D_A)}=(-1)^{\sum_{j=1}^{\lambda_1}\frac{(j-1)\lambda_j'(\lambda_j'-1)}{2}}$.
Recalling~\eqref{sec5_eq_priliminary-weight-coefficient-formula} it is then enough to prove that $\gamma_{A,D_A}=\gamma_{A,C}$ iff $D_A=C$, for all $C\in \mathbb{T}(\lambda_A,p)$.
One implication is trivial. To prove the other we suppose for contradiction that there exists $C\in \mathbb{T}(\lambda_A,p)$ such that $\gamma_{A,D_A}=\gamma_{A,C}$ and~$D_A\neq C$. We~shall use the notation~$\gamma=\gamma_{A,D_A}$ and~$\zeta=\gamma_{A,C}$ for the remainder of the proof of~$(a)$.

The assumption that $\gamma=\zeta$ together with the formula~\eqref{sec5_eq_exponent-matrix-2} and the fact that $\gamma$ is lower triangular, implies that
\begin{gather}
\label{sec5_eq_entry-inequility}
a_{k,l}\geq c_{k,l},
	\end{gather}
for all $(k,l)\in\lambda_A$.

Recall that the $(k,l)'th$ entry of~$D_A$ is $k$, for all $(k,l)\in\lambda_A$. The second assumption, namely that $D_A\neq C$, implies that there exists $(k_0,l_0)\in \lambda_A$ such that $c_{k_0,l_0}\neq k_0$. Let $s=a_{k_0,l_0}$, then among the possible options we may chose $(k_0,l_0)$ such that $s$ is as small as possible and~$k_0$ is as large as possible. In~other words $(k_0,l_0)$ is chosen such that $c_{k,l}=k$ for all $(k,l)\in\lambda_A$ with $a_{k,l}<s$, or with $a_{k,l}=s$ and~$k>k_0$.
Since the entries in the $l_0$'th column of~$C$ are distinct, our choice of~$(k_0,l_0)$ together with~\eqref{sec5_eq_entry-inequility} imply that $c_{k_0,l_0}=t$, for some $t\in \{k_0+1,\dots,s\}$.
Furthermore, since $t>k_0$, our choice of~$(k_0,l_0)$ also implies that $c_{t,l}=t$ for all $l\in\{1,\dots,(\lambda_A)_t\}$ with $a_{t,l}=s$. By~\eqref{sec5_eq_exponent-matrix-leading-term}, we have thus found $1+\gamma_{s,t}$ distinct coordinates $(k,l)\in\lambda_A$ such that $a_{k,l}=s$ and~$c_{k,l}=t$. By~\eqref{sec5_eq_exponent-matrix-2} this means that $\zeta_{s,t}\geq 1+\gamma_{s,t}$ and thus that $\gamma\neq \zeta$. This contradicts our initial assumption, finally proving that $\gamma_{A,D_A}=\gamma_{A,C}$ iff $D_A=C$.

To prove statement $(b)$, let $A,B\in \mathbb{Y}_n(p)$ with $A<B$. By~\eqref{sec5_eq_priliminary-weight-coefficient-formula} it is enough to prove that $\gamma_{B,D_B}\neq\gamma_{A,C}$, for all $C\in \mathbb{T}(\lambda_A,p)$. This statement is trivial if $\mu_B=\mu_A$. We~assume therefore that $A$ and~$B$ have the same weight and consider an arbitrary $C\in\mathbb{T}(\lambda_A,p)$. For the remainder of this proof we shall use the notation~$\gamma=\gamma_{B,D_B}$ and~$\zeta=\gamma_{A,C}$.
By Definition~\ref{sec5_defi_ordering} the assumption~$A<B$ implies that there exists $s,t\in\{1,\dots,n\}$ such that
\begin{gather*}
\lambda_{A^i}=\lambda_{B^i},
	\qquad
	(\lambda_{A^{s}})_{\alpha}=(\lambda_{B^{s}})_{\alpha}
	\qquad\text{and}\qquad
	(\lambda_{A^{s}})_{t}<(\lambda_{B^{s}})_{t},
	\end{gather*}
for all $i<s$ and~$\alpha<t$.
Using this together with~\eqref{sec5_eq_exponent-matrix-leading-term} and the fact that in each column of~$C$ the entries are mutually distinct, we can make the estimate
\begin{gather*}
\sum_{i\leq s,\alpha\leq t} \zeta_{i\alpha}
	\leq\sum_{\alpha\leq t} (\lambda_{A^s})_\alpha
	<\sum_{\alpha\leq t} (\lambda_{B^s})_\alpha
	= \sum_{i\leq s,\alpha\leq t} \gamma_{i\alpha}.
	\end{gather*}
This clearly implies $\gamma\neq\zeta$, which proves $(b)$.
	\end{proof}

One should note that $D_A$ is not the only c.d.\ Young tableau in $\mathbb{T}(\lambda_A, p)$ with the properties described in Proposition~\ref{sec5_prop_leading-term}. These properties are satisfied by any tableau in $\mathbb{T}(\lambda_A, p)$ for which all entries of any single row are equal.

\begin{Corollary}
\label{sec5_coro_increasing-basis}
For any $B\in \mathbb{Y}_n(p)$, $\omega_B(p)\neq 0$ and~$\omega_B(p)\notin \operatorname{span}_\mathbb{C}\{\omega_A(p)\colon A<B\}$.
	\end{Corollary}

We can now finally prove that the vectors defined in Definition~\ref{sec4_defi_basis-elements}, with properties described in Lemma~\ref{sec4_lemm_basis-vector-properties}, form bases for the modules $\overline{V}_n(p)$ and~$L_n(p)$.
\begin{Theorem}
\label{sec5_theo_basis}
The modules $\overline{V}_n(p)$ and~$L_n(p)$ have bases
\begin{gather*}
\{v_A(p)\colon A\in \mathbb{Y}_n\}\qquad\text{and}\qquad \{\omega_A(p) \colon A\in \mathbb{Y}_n(p)\}.
	\end{gather*}
	\end{Theorem}

\begin{proof}
We begin with $\{\omega_A(p)\colon A\in \mathbb{Y}_n(p)\}$. It is enough to prove that we have bases for each of the weight spaces of~$L_n(p)$. That is, it is enough to prove that, for any $\mu\in\mathbb{N}_0^n$, the set
\begin{gather*}
\{\omega_A(p)\colon A\in \mathbb{Y}_n(p),\ \mu_A=\mu\}
	\end{gather*}
is a basis for~$L_n(p)_{\mu+\frac{p}{2}}$. By Lemma~\ref{sec4_lemm_basis-vector-properties} they have the right weight. Since the set is finite, linear independence follows from Corollary~\ref{sec5_coro_increasing-basis}.
Lemma~\ref{sec4_lemm_dim-weight-spaces} and~\eqref{sec4_eq_kostka} tells us that the cardinality of the set $\{\omega_A(p)\colon A\in \mathbb{Y}_n(p),\ \mu_A=\mu\}$ agrees with the dimension of the weight space, proving that we indeed have a basis.

To prove that $\{v_A(p)\colon A\in \mathbb{Y}_n\}$ is a basis for~$\overline{V}_n(p)$ we consider $q\in\mathbb{N}$, with $q\geq n$. The~comments following Lemma~\ref{sec3_lemm_kernel-equals-maximal-submodule} tell us that $\Psi_q\colon \overline{V}_n(q)\to L_n(q)$ is an isomorphism of~$\mathfrak{osp}(1|2n)$-modules. This together with Lemma~\ref{sec3_lemm_basis-positive-enveloping-algebra} implies that the composition of~$\Psi_q$ and~$\Phi_q$,
\begin{gather*}
\Psi_q\circ\Phi_q\colon\ U\big(\mathfrak{osp}(1|2n)^+\big)\rightarrow L_n(q),
	\end{gather*}
is an isomorphism of vector spaces.
Following the notation used in Definition~\ref{sec4_defi_basis-elements} we define elements
\begin{gather*}
B_A^+:=B_{a_m}^+\cdots B_{a_1}^+\in U\big(\mathfrak{osp}(1|2n)^+\big),
	\end{gather*}
for each $A\in \mathbb{Y}_n$, where $m=\ell(\lambda_A')$.
Noting that $\Psi_q\circ\Phi_q\big(B_A^+\big)=\omega_A(q)$, for all $A\in\mathbb{Y}_n$, it follows that the set $\{B_A^+\colon A\in\mathbb{Y}_n\}$ is a basis for~$U\big(\mathfrak{osp}(1|2n)^+\big)$ and thus by Lemma~\ref{sec3_lemm_basis-positive-enveloping-algebra} that
\begin{gather*}
\big\{v_A(p)=B_A^+|0\rangle=\Phi_p\big(B_A^+\big)\colon A\in \mathbb{Y}_n\big\},
	\end{gather*}
is a basis for~$\overline{V}_n(p)$ regardless of the value of~$p$.
	\end{proof}

The proof of this Theorem yields the following corollary.
\begin{Corollary}
\label{sec5_coro_basis-for-creation-subalgebra}
The set
\begin{gather*}
\big\{B_A^+:=B_{a_m}^+\cdots B_{a_1}^+\colon A\in\mathbb{Y}_n\big\},
	\end{gather*}
where $m=\ell(\lambda_A')$, for each $A$, forms a basis for~$U\big(\mathfrak{osp}(1|2n)^+\big)$.
	\end{Corollary}

\section[Action of osp(1|2n) on tableau vectors omega A(p)]{Action of~$\boldsymbol{\mathfrak{osp}(1|2n)}$ on tableau vectors $\boldsymbol{\omega_A(p)}$}\label{sec6}

In this section we will obtain formulas for the action of the $\mathfrak{osp}(1|2n)$ generators $X_i$ and~$D_i$ on the basis for~$L_n(p)$ constructed in Definition~\ref{sec4_defi_basis-elements}.
More precisely, considering the normalization
\begin{gather*}
\tilde{\omega}_A(p):= \frac{1}{\lambda_A!}\omega_A(p),
	\end{gather*}
for all $A\in \mathbb{Y}_n(p)$, we want to obtain coefficients $\bar{c}_B(i,p,A)$ and~$\hat{c}_B(i,p,A)$, for all $i\in \{1,\dots,n\}$ and~$A\in\mathbb{Y}_n(p)$, such that
\begin{gather}
X_i\tilde{\omega}_A(p)=\sum_{B\in \mathbb{Y}_n(p)} \bar{c}_B(i,p,A) \tilde{\omega}_B(p),
\\
D_i\tilde{\omega}_A(p)=\sum_{B\in \mathbb{Y}_n(p)} \hat{c}_B(i,p,A) \tilde{\omega}_B(p).
\label{sec6_eq_action-initial-expansion}
	\end{gather}
Regarding the actions of~$B_i^+$ and~$B_i^-$ on the basis vectors $\tilde{v}_A(p)=\frac{1}{\lambda_A!}v_A(p)$ of~$\overline{V}_n(p)$, the corresponding expansion coefficients can be calculated from the coefficients in~\eqref{sec6_eq_action-initial-expansion}. The details of this will be discussed at the end of this section after we obtain formulas for the calculation of the coefficients $\bar{c}_B(i,p,A)$ and~$\hat{c}_B(i,p,A)$, see Proposition~\ref{sec6_prop_verma-matrix-elements}.

Before producing general formulas for obtaining the expansions in~\eqref{sec6_eq_action-initial-expansion}, we take a look at a~few cases where the actions of~$X_i$ and~$D_i$ on~$\tilde{\omega}_A(p)$ are particularly simple.
First, if $(\mu_A)_i=0$ then $D_i\tilde{\omega}_A(p)=0$. Second, if $i$ is greater than or equal to the topmost entry of the rightmost column of~$A$, then $X_i\tilde{\omega}_A(p)=\tilde{\omega}_{B}(p)$, where $B$ is the tableau obtained by adding a single column containing only the box $\ytableausetup{smalltableaux}\ytableaushort{i}$ to the right side of~$A$.
For example, if $A=\ytableausetup{smalltableaux}\ytableaushort{23,4}$, then
\begin{gather*}
\ytableausetup{boxsize=0.9em}
X_i\tilde{\omega}\,{}_{\tiny{\ytableaushort{23,4}}}(p)
=\tilde{\omega}\,{}_{\tiny{\ytableaushort{23i,4}}}(p),
	\end{gather*}
for all $i\geq 3$.

The vectors $X_i\tilde{\omega}_A(p)$ and~$D_i\tilde{\omega}_A(p)$ are weight vectors, for all $i\in\{1,\dots,n\}$ and~$A\in \mathbb{Y}_n(p)$, specifically
\begin{gather*}
X_i\tilde{\omega}_A(p) \in L_n(p)_{\mu_A+\epsilon_i+\frac{p}{2}}
\qquad\text{and}\qquad
D_i\tilde{\omega}_A(p)	\in L_n(p)_{\mu_A-\epsilon_i+\frac{p}{2}}.
	\end{gather*}
This fact leads to the following observation
\begin{Remark}
\label{sec6_rema_vanishing-coefficients}
For all $i\in \{1,\dots,n\}$ and~$A,B\in\mathbb{Y}_n(p)$,
\begin{gather*}
\bar{c}_B(i,p,A)=0\qquad \text{if}\quad \mu_B \neq \mu_A+\epsilon_i
	\end{gather*}
and
\begin{gather*}
\hat{c}_B(i,p,A)=0\qquad \text{if}\quad \mu_B \neq\mu_A-\epsilon_i.
	\end{gather*}
	\end{Remark}
Given a weight $\mu\in \mathbb{N}_0^n$ we let $d_\mu$ be the dimension of the weight space $L_n(p)_{\mu+\frac{p}{2}}$, see~\eqref{sec4_eq_dim-weight-spaces-irrep},
\begin{gather*}
d_\mu:= \dim L_n(p)_{\mu+\frac{p}{2}}.
	\end{gather*}
Remark~\ref{sec6_rema_vanishing-coefficients} then tells us that $X_i\tilde{\omega}_A(p)$ and~$D_i\tilde{\omega}_A(p)$ are linear combinations of at most $d_{\mu_A+\epsilon_i}$ and~$d_{\mu_A-\epsilon_i}$ tableau vectors respectively.

The remaining coefficients in~\eqref{sec6_eq_action-initial-expansion} can be obtained by solving the system of linear equations that comes from comparing the monomial expansions, Lemma~\ref{sec5_lemm_priliminary-weight-coefficient-formula}, of~$X_i\tilde{\omega}_A$ and~$D_i\tilde{\omega}_A$ with those of~$\tilde{\omega}_B(p)$ for~$\mu_B=\mu_A+\epsilon_i$ and~$\mu_B=\mu_A-\epsilon_i$ respectively.
This process is very inefficient as it entails calculating all the coefficients in the monomial expansion of the involved vectors.
Our approach will be to reduce the number of coefficients we need to determine to only the necessary ones and subsequently to find formulas for calculating them.

Given $A\in\mathbb{Y}_n(p)$ and~$C\in\mathbb{T}(\lambda_A,p)$ the coefficient of~$x_{A,C}e_C$ in the monomial expansion of a~given vector $v\in L_n(p)$ is equal to
\begin{gather*}
\langle x_{A,C}e_C,v \rangle.
	\end{gather*}
Given $\mu\in\mathbb{N}_0^n$ we denote the $d_\mu$ s.s.\ Young tableaux $A_1,\dots,A_{d_\mu}$ of weight $\mu$ in such a way that
\begin{gather*}
A_1<\cdots<A_{d_\mu}.
	\end{gather*}
Define the $d_\mu\times d_\mu$ matrix $U_\mu$ and the vectors $f_\mu(v)\in \mathbb{C}^{d_\mu}$, for all $v\in L_n(p)_{\mu+\frac{p}{2}}$, as follows
\begin{gather*}
(U_\mu)_{k,l}:=\langle x_{A_k,D_{A_k}}e_{D_{A_k}}, \tilde{\omega}_{A_l}(p)\rangle
	\end{gather*}
and
\begin{gather*}
(f_\mu)_k(v):=\langle x_{A_k,D_{A_k}}e_{D_{A_k}}, v\rangle,
	\end{gather*}
for all $k,l\in\{1,\dots,d_\mu\}$.
The $(k,l)$'th entry of the matrix $U_\mu$ is thus the coefficient of the leading term $x_{A_k,D_{A_k}}e_{D_{A_k}}$ of~$\tilde{\omega}_{A_k}(p)$ as it appears in $\tilde{\omega}_{A_l}(p)$.

\begin{Proposition}
\label{sec6_prop_priliminary-formula}
For any $\mu\in\mathbb{N}_0^n$, the matrix $U_\mu\in M_{d_\mu,d_\mu}(\mathbb{Z})$ is integer and upper unitriangular.
Furthermore, for any $v\in L_n(p)_{\mu+\frac{p}{2}}$, we have
\begin{gather*}
v	=\sum_{k=1}^{d_\mu} \big(U_{\mu}^{-1}\cdot f_\mu(v)\big)_{k}\tilde{\omega}_{A_k}(p)
\\ \hphantom{v}
{}=\sum_{k=1}^{d_\mu}	\sum_{t=0}^{d_\mu-k}	\sum_{k=l_0<\cdots<l_t\leq d_\mu}
(-1)^t	\prod_{s=1}^t	(U_\mu)_{l_{s-1},l_s}	(f_\mu)_{l_t}	\tilde{\omega}_{A_k}(p).
\end{gather*}
\end{Proposition}
Since $X_i\tilde{\omega}_A(p)$ and~$D_i\tilde{\omega}_A(p)$ are weight vectors this result provides formulas for calcu\-la\-ting the coefficients $\bar{c}_B(i,p,A)$ and~$\hat{c}_B(i,p,A)$, only dependent on~$U_{\mu_A+\epsilon_1}$ and~$f_{\mu_A+\epsilon_i}(X_i\tilde{\omega}_A(p))$, and~$U_{\mu_A-\epsilon_1}$ and~$f_{\mu_A-\epsilon_i}(D_i\tilde{\omega}_A(p))$, respectively.

In Lemma~\ref{sec6_lemm_formulas-matrix-elements}, we will see that the entries of the matrices $U_{\mu_A\pm\epsilon_1}$ and vectors $f_{\mu_A+\epsilon_i}(X_i\tilde{\omega}_A(p))$ and~$f_{\mu_A-\epsilon_i}(D_i\tilde{\omega}_A(p))$, are algebraic expressions in coefficients $c_B(\gamma)$.
Recalling Lemma~\ref{sec5_lemm_priliminary-weight-coefficient-formula}, this means that the coefficients $\bar{c}_B(i,p,A)$ and~$\hat{c}_B(i,p,A)$, also known as the generator matrix elements, are integers.
	
\begin{proof}
The expansion in Lemma~\ref{sec5_lemm_monomial-expansion} implies that the matrix $U_\mu$ is integer, and Corollary~\ref{sec5_coro_increasing-basis} implies that it is upper unitriangular.
Together Corollary~\ref{sec5_coro_increasing-basis} and Theorem~\ref{sec5_theo_basis} gives
\begin{gather*}
\langle x_{A_k,D_{A_k}}e_{D_{A_k}}, v\rangle
=\sum_{l=1}^{d_\mu}\langle x_{A_k,D_{A_k}}e_{D_{A_k}}, \tilde{\omega}_{A_l}(p)\rangle
\langle \tilde{\omega}_{A_l}(p), v\rangle
\\ \hphantom{\langle x_{A_k,D_{A_k}}e_{D_{A_k}}, v\rangle}
{}=\langle \tilde{\omega}_{A_k}(p), v\rangle	+\sum_{l=k+1}^{d_\mu}
	\langle x_{A_k,D_{A_k}}e_{D_{A_k}}, \tilde{\omega}_{A_l}(p)\rangle
	\langle \tilde{\omega}_{A_l}(p), v\rangle.
\end{gather*}
By moving the sum to the other side of the equation and substituting the inner products by definitions of~$U_\mu$ and~$f_\mu(v)$, we get
\begin{gather}
\label{sec6_eq_backward-substitution-formula}
\langle \tilde{\omega}_{A_k}(p), v\rangle
	= (f_\mu)_k(v)
	- \sum_{l=k+1}^{d_\mu}
	(U_\mu)_{k,l}
	\langle \tilde{\omega}_{A_l}(p), v\rangle.
	\end{gather}
One may recognize this as the backward substitution formula for calculating linear equations of the form $Ax=b$, where $A$ is a upper unitriangular matrix. Whence,
\begin{gather*}
v = \sum_{k=1}^{d_\mu} \big(U_{\mu}^{-1}\cdot f_\mu(v)\big)_{k}\tilde{\omega}_{A_k}(p).
	\end{gather*}
Repeated use of~\eqref{sec6_eq_backward-substitution-formula} shows that
\begin{gather*}
v=	\sum_{k=1}^{d_\mu}
	\sum_{t=0}^{d_\mu-k}
	\sum_{k=l_0<\cdots<l_t\leq d_\mu}
	(-1)^t
	\prod_{s=1}^t
	(U_\mu)_{l_{s-1},l_s}
	(f_\mu)_{l_t}
	\tilde{\omega}_{A_k}(p).
\tag*{\qed}
\end{gather*}
\renewcommand{\qed}{}
\end{proof}

The following is a sketch of how to determine the expansion of~$X_i\tilde{\omega}_A(p)$ using Proposition~\ref{sec6_prop_priliminary-formula}. We~do not include calculations of the entries of~$U_\mu$ and~$f_\mu(X_i\tilde{\omega}_A(p))$.
Consider the case $n=4$, $p=2$, $i=1$ and
\begin{gather*}
\ytableausetup{centertableaux,boxsize=1.1em}
A=\ytableaushort{23,4}
	\end{gather*}
Then $X_1\tilde{\omega}\,{}_{\tiny{\ytableausetup{boxsize=0.9em}\ytableaushort{23,4}}}(2)$ is a linear combination of the vectors $\tilde{\omega}_B(2)$, for~$B\in \mathbb{Y}_4(2)$ of weight $\mu=\mu_A+\epsilon_1=(1,1,1,1)$. Since $d_{\mu}=6$, there are $6$ s.s.\ Young tableaux of weight $(1,1,1,1)$. These are, in increasing order,
\begin{gather*}
\ytableausetup{centertableaux,boxsize=1.1em}
A_1=\ytableaushort{13,24}\,,\qquad
A_2=\ytableaushort{134,2},\,\qquad
A_3=\ytableaushort{12,34}\,,\qquad
A_4=\ytableaushort{124,3}\,,
\\[1ex]
 A_5=\ytableaushort{123,4}\,,\qquad
 A_6=\ytableaushort{1234}\,.
	\end{gather*}
Had we instead chosen $p\geq 4$, then we would also have to take into account the $4$ remaining s.s.\ Young tableaux of weight $(1,1,1,1)$: $\ytableausetup{boxsize=0.9em,smalltableaux}\ytableaushort{1,2,3,4}$\,, $\ytableausetup{boxsize=0.9em,smalltableaux}\ytableaushort{14,2,3}$, $\ytableausetup{boxsize=0.9em,smalltableaux}\ytableaushort{13,2,4}$ and~$\ytableausetup{boxsize=0.9em,smalltableaux}\ytableaushort{14,2,3}$. When $p=2$ no basis vectors correspond to these tableaux as they have strictly more than $2$ rows.

Calculating the coefficients of the relevant terms we get
\begin{gather*}
U_{\mu}=
\begin{pmatrix}
 1 & -1 & 0 & 0 & 1 & 1 \\
 0 & 1 & 0 & 0 & 0 & -1 \\
 0 & 0 & 1 & -1 & -1 & -1 \\
 0 & 0 & 0 & 1 & 0 & 1 \\
 0 & 0 & 0 & 0 & 1 & -1 \\
 0 & 0 & 0 & 0 & 0 & 1
	\end{pmatrix}
\qquad\text{and}\qquad
f_{\mu}\big(X_1\tilde{\omega}\,{}_{\tiny{\ytableausetup{boxsize=0.9em}\ytableaushort{23,4}}}(2)\big)=
\begin{pmatrix}
0 \\
-1 \\
1 \\
0 \\
1 \\
0 \\
	\end{pmatrix}\!.
	\end{gather*}
Note that in general the entries of~$U_{\mu}$ take values in all of~$\mathbb{Z}$ and not just in $\{\pm 1,0\}$.
Calculating the inverse of~$U_{\mu}$ we get
\begin{gather*}
U_{\mu}^{-1}=
\begin{pmatrix}
1 & 1 & 0 & 0 & -1 & -1 \\
0 & 1 & 0 & 0 & 0 & 1 \\
0 & 0 & 1 & 1 & 1 & 1 \\
0 & 0 & 0 & 1 & 0 & -1 \\
0 & 0 & 0 & 0 & 1 & 1 \\
0 & 0 & 0 & 0 & 0 & 1
	\end{pmatrix}
	\end{gather*}
and
\begin{gather*}
U_{\mu}^{-1}\cdot f_{\mu}\Big(X_1\tilde{\omega}\,{}_{\tiny{\ytableausetup{boxsize=0.9em}\ytableaushort{23,4}}}(2)\Big) = (-2,-1,2,0,1,0).
	\end{gather*}
Using Proposition~\ref{sec6_prop_priliminary-formula} this tells us that
\begin{gather*}
X_1\tilde{\omega}\,{}_{\tiny{\ytableausetup{boxsize=0.9em}\ytableaushort{23,4}}}(2)
=
-2\tilde{\omega}\,{}_{\tiny{\ytableausetup{boxsize=0.9em}\ytableaushort{13,24}}}(2)
-\tilde{\omega}\,{}_{\tiny{\ytableausetup{boxsize=0.9em}\ytableaushort{134,2}}}(2)
+2\tilde{\omega}\,{}_{\tiny{\ytableausetup{boxsize=0.9em}\ytableaushort{12,34}}}(2)
+ \tilde{\omega}\,{}_{\tiny{\ytableausetup{boxsize=0.9em}\ytableaushort{123,4}}}(2).
	\end{gather*}

To determine general formulas for the entries of~$U_\mu$, $f_\mu(X_i\tilde{\omega}_A(p))$ and~$f_\mu(D_i\tilde{\omega}_A(p))$ we first write them as functions of the coefficients $c_B(\gamma)$ introduced right before Lemma~\ref{sec5_lemm_priliminary-weight-coefficient-formula}. Simple calculations and applications of definitions yield the following formulas.

\begin{Lemma}
\label{sec6_lemm_formulas-matrix-elements}
Let $A,B\in\mathbb{Y}_n(p)$ and~$i\in\{1,\dots,n\}$. Then
\begin{gather*}
\langle x_{B,D_B}e_{D_B}, \tilde{\omega}_A(p)\rangle
= c_B(\gamma_{B,D_B})	c_A(\gamma_{B,D_B}),
\\
\langle x_{B,D_B}e_{D_B}, X_i\tilde{\omega}_A(p)\rangle
=c_B(\gamma_{B,D_B})\sum_{\alpha=1}^p \prod_{\beta=1}^{\alpha-1}(-1)^{(\lambda_B)_\beta}
			c_A(\gamma_{B,D_B}-\epsilon_{i,\alpha}),
\\
\langle x_{B,D_B}e_{D_B}, D_i\tilde{\omega}_A(p)\rangle
=c_B(\gamma_{B,D_B})\sum_{\alpha=1}^p \prod_{\beta=1}^{\alpha-1}(-1)^{(\lambda_B)_\beta}
((\gamma_{B,D_B})_{i,\alpha}+1)c_A(\gamma_{B,D_B}+\epsilon_{i,\alpha}),
\end{gather*}
where $c_A(\gamma_{B,D_B}-\epsilon_{i,\alpha}):=0$ if $(\gamma_{B,D_B})_{i,\alpha}=0$, and where
$\epsilon_{i,\alpha}\in M_{n,p}(\mathbb{N}_0)$ is the matrix with
\begin{gather*}
(\epsilon_{i,\alpha})_{j,\beta}:=\delta_{i,j}\delta_{\alpha,\beta},
\end{gather*}
for all $i,j\in\{1,\dots,n\}$ and~$\alpha,\beta\in\{1,\dots,p\}$.
\end{Lemma}

Lemma~\ref{sec5_lemm_priliminary-weight-coefficient-formula} already gives a formula for~$c_B(\gamma_{B,D_B})$.
The rest of this section will be dedicated to obtaining concrete formulas for the calculation of~$c_A(\gamma)$ given any $A\in \mathbb{Y}_n(p)$ and~$\gamma\in M_{n,p}(\mathbb{N}_0)$.
Consider first the expression given in Lemma~\ref{sec5_lemm_priliminary-weight-coefficient-formula}
\begin{gather*}
c_A(\gamma)=\sum_{\substack{C\in\mathbb{T}(\lambda_A,p),\\ \gamma=\gamma_{A,C}}} (-1)^{N(C)},
	\end{gather*}
where $N(C)\in \mathbb{N}_0$ is a number such that
\begin{gather*}
x_{A,C}e_C =(-1)^{N(
C)}x^{\gamma_{A,C}} e^{\eta_{\gamma_{A,C}}}.
	\end{gather*}
The general idea behind our approach to the calculation of~$c_A(\gamma)$ is to determine the set
\begin{gather*}
\{C\in\mathbb{T}(\lambda_A,p)\colon \gamma_{A,C}=\gamma\},
	\end{gather*}
and for each element in this set calculate the number $(-1)^{N(C)}$.
Consider the following more general class of Young tableaux,
\begin{gather*}
\mathbb{E}(\lambda,p):=
\{\text{fillings of the Young diagram of shape $\lambda\in \mathcal{P}$ by numbers $1,\dots,p$}\}.
	\end{gather*}
We shall refer to the elements of~$\mathbb{E}(\lambda,p)$ as Young tableaux. Note that $\mathbb{T}(\lambda,p)\subset \mathbb{E}(\lambda,p)$.
The definitions~\eqref{sec5_eq_clifford-monomial-def} and~\eqref{sec5_eq_monomial-def} can naturally be extended to hold, for any $A\in\mathbb{Y}_n(p)$ with $\lambda=\lambda_A$ and~$T\in \mathbb{E}(\lambda,p)$, by letting $m=(\lambda_A)_1$ and defining
\begin{gather*}
e_T:=\big(e_{t_{1,m}}\cdots e_{t_{\lambda_m',m}}\big)\cdots \big(e_{t_{1,1}}\cdots e_{t_{\lambda_1',1}}\big)\in \mathcal{C}\ell_p
\end{gather*}
and
\begin{gather*}
x_{A,T}:= \big(x_{a_{1,m},t_{1,m}}\cdots x_{a_{\lambda_m',m},t_{\lambda_m',m}}\big)\cdots \big(x_{a_{1,1},t_{1,1}}\cdots x_{a_{\lambda_1',1},t_{\lambda_1',1}}\big)\in \mathbb{C}[\mathbb{R}^{np}].
	\end{gather*}	
Let furthermore $\gamma_{A,T}\in M_{n,p}(\mathbb{N}_0)$ and~$N(T)$ such that $x_{A,T}e_T=(-1)^{N(T)}x^{\gamma_{A,T}} e^{\eta_{\gamma_{A,T}}}$.
Define now, for each $\mu\in\mathbb{N}_0^n$, the following permutation group,
\begin{gather*}
S_\mu := S_{\mu_1} \times \cdots \times S_{\mu_n}.
	\end{gather*}	
The reason for introducing these objects is that the set $\mathbb{E}(\lambda,p)$ carries a useful action of the permutation group $S_{\mu}$, for each $A\in\mathbb{Y}_n(p)$ with $\mu_A=\mu$ and~$\lambda_A=\lambda$. To define this action, let
\begin{gather*}
y_A(i,s)=(k_A(i,s),l_A(i,s))\in\lambda_A,
	\end{gather*}
for~$i\in\{1,\dots,n\}$ and~$s\in\{1,\dots,(\mu_A)_i\}$, be the coordinates of~$\lambda_A$ such that
\begin{gather*}
\lambda_{A^i} - \lambda_{A^{i-1}} = \big\{ y_A(i,s)\colon s\in\{1,\dots,(\mu_A)_i\}\big\},
	\end{gather*}
and~$l_A(i,1)> \dots >l_A(i,(\mu_A)_i)$. Here $\lambda_{A^i} - \lambda_{A^{i-1}}$ is the set of coordinates $y\in\lambda_A$ whose corresponding entries $a_y$ in $A$ are $i$.
Specifically
\begin{gather*}
l_A(i,s)=\max\Big\{l\in\{1,\dots,(\lambda_{A^i})_1\}\colon \textstyle s=\sum\limits_{r=l}^{(\lambda_{A^i})_1} (\lambda_{A^i})'_{r}-(\lambda_{A^{i-1}})'_{r}\Big\}\!,
\\
k_A(i,s)=(\lambda_{A^i})_{l_A(i,s)}'.
\end{gather*}
Note that these coordinates cover all entries of~$\lambda_A$
\begin{gather*}
\lambda_A=\big\{y_A(i,s)\colon i\in\{1,\dots,n\}, s\in\{1,\dots,(\mu_A)_i\}\big\}.
	\end{gather*}
In the terminology of Macdonald \cite{Macdonald-2015}, $\{y_A(i,s)\colon s\in\{1,\dots,(\mu_A)_i\}\}$ is the $i$'th horizontal strip of~$A$ consisting of the boxes $\ytableausetup{smalltableaux}\ytableaushort{i}$ and indexed from right to left by $s$. As~an example consider
\begin{gather*}
\ytableausetup{centertableaux,boxsize=1.1em}
A=\ytableaushort{11122,223,3}
	\end{gather*}
Then $\mu_1=3$, $\mu_2=4$ and~$\mu_3=2$. For $i=2$, the coordinates $y_A(2,s)\in \lambda_A$ for~$s\in\{1,2,3,4\}$ are given by
\begin{gather*}
y_A(2,1)=(1,5), \qquad
y_A(2,2)=(1,4), \qquad
y_A(2,3)=(2,2)\qquad\text{and}\qquad y_A(2,4)=(2,1).
\end{gather*}
Visually these coordinates refer to the gray boxes in the Young diagram $\lambda_A$:
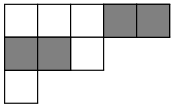
\begin{gather*}
\ytableausetup{boxsize=1.1em}
\ydiagram[*(gray) ]
{3+2,2}
*[*(white)]{5,3,1}
	\end{gather*}
The action
\begin{gather*}
\pi_A\colon\ S_{\mu_A}\times \mathbb{E}(\lambda_A,p)\to \mathbb{E}(\lambda_A,p),
	\end{gather*}
is defined by letting $T^{\sigma,A}:=\pi_A(\sigma,T)$ be the Young tableau with entries $t^{\sigma,A}_y$, for~$y\in \lambda_A$, defined~by
\begin{gather}
\label{sec6_eq_def-permutation-action}
t^{\sigma,A}_{y_A(i,s)}:= t_{y_A(i,\sigma^{-1}_i(s))},
	\end{gather}
for all $i\in\{1,\dots,n\}$ and~$s\in\{1,\dots,(\mu_A)_i\}$.
Consider $A$ as in the previous example, and let $p=4$ and
\begin{gather*}
\ytableausetup{centertableaux,boxsize=1.1em}
T=\ytableaushort{21122,433,3}\in\mathbb{T}((5,3,1),4).
	\end{gather*}
Let $I$ denote the identity permutation, and let $\sigma,\tau\in S_{\mu_A}=S_3\times S_4\times S_2$ with
\begin{gather*}
\sigma=(I,(13),I) \qquad\text{and}\qquad \tau=(I,(2413),I).
	\end{gather*}
The $\pi_A$ action of these permutations only permutes the entries of~$T$ corresponding to the coor\-dinates of the $2$nd horizontal strip of~$A$. The boxes corresponding to these coordinates are marked with gray below:
\begin{gather*}
T^{\sigma,A}=
\ytableausetup{centertableaux,boxsize=1.1em}
\begin{ytableau}
 2& 1 & 1 &*(gray) 2 &*(gray) 3 \\
*(gray) 4 &*(gray) 2 & 3 \\
 3
\end{ytableau}
\qquad\text{and}\qquad
T^{\tau,A}=
\ytableausetup{centertableaux,boxsize=1.1em}
\begin{ytableau}
 2& 1 & 1 &*(gray) 3 &*(gray) 4 \\
*(gray) 2 &*(gray) 2 & 3 \\
 3
\end{ytableau}
	\end{gather*}
Note here that $T^{\sigma,A}\in\mathbb{T}((5,3,1),4)$, whereas $T^{\tau,A}\notin \mathbb{T}((5,3,1),4)$.

\begin{Lemma}
\label{sec6_lemm_permutation-action}
Let $A\in\mathbb{Y}_n(p)$ and~$T,T'\in \mathbb{E}(\lambda_A,p)$, then $x_{A,T}=x_{A,T'}$ if and only if there exists $\sigma\in S_{\mu_A}$ such that $T'=T^{\sigma,A}$.
	\end{Lemma}
Equivalently this lemma states that
\begin{gather*}
\big\{T'\in\mathbb{E}(\lambda_A,p)\colon \gamma_{A,T'}=\gamma_{A,T}\big\}
=\big\{T^{\sigma,A}\colon \sigma \in S_{\mu_A}\big\}.
	\end{gather*}
\begin{proof}
Let $a_y$, $t_y$ and~$t_y'$, for~$y\in\lambda_A$, denote the entries of~$A$, $T$ and~$T'$ respectively.
Note first that, for all $\sigma=(\sigma_1,\dots,\sigma_n)\in S_{\mu_A}$,
\begin{gather*}
x_{A,T} = \prod_{i=1}^n \prod_{s=1}^{(\mu_A)_i} x_{a_{y_A(i,s)},t_{y_A(i,s)}}
=\prod_{i=1}^n \prod_{s=1}^{(\mu_A)_i} x_{a_{y_A(i,s)},t_{y_A(i,\sigma^{-1}_i(s))}}
=x_{A,T^{\sigma,A}}.
	\end{gather*}
This yields one of the implications stated in the lemma. If $x_{A,C}=x_{A,C'}$ then, for all $i\in\{1,\dots,n\}$,
\begin{gather*}
\prod_{s=1}^{(\mu_A)_i} x_{a_{y_A(i,s)},t_{y_A(i,s)}} =\prod_{s=1}^{(\mu_A)_i} x_{a_{y_A(i,s)},t'_{y_A(i,s)}}.
	\end{gather*}
This implies that the vectors $\big(t_{y_A(i,1)},\dots, t_{y_A(i,(\mu_A)_i)}\big)$ and~$\big(t'_{y_A(i,1)},\dots, t'_{y_A(i,(\mu_A)_i)}\big)$ contain the same amount of~$\alpha$-entries, for all $\alpha\in \{1,\dots,p\}$, letting us construct a permutation~$\sigma_i\in S_{(\mu_A)_i}$ such that
\begin{gather*}
t'_{y_A(i,s)}=t_{y_A(i,\sigma^{-1}_i(s))},
	\end{gather*}
for all $s\in\{1,\dots,(\mu_A)_i\}$. Repeating this process for all $i\in\{1,\dots,n\}$ and defining $\sigma:=(\sigma_1,\dots,\sigma_n)$, it is clear that $\sigma \in S_{\mu_A}$ and~$T'=T^{\sigma,A}$.
	\end{proof}

A Young tableau $T_{\gamma,A}\in \mathbb{E}(\lambda_A,p)$ for which $\gamma_{A,T_{\gamma,A}}=\gamma$ can be constructed in the following way. For all $\alpha\in\{1,\dots,p\}$, $i\in\{1,\dots,n\}$ and~$s$ with $\sum_{\beta=1}^{\alpha-1}\gamma_{i,\beta}< s \leq \sum_{\beta=1}^{\alpha}\gamma_{i,\beta}$ we define the entry of~$T_{\gamma,A}$ at the coordinate $y_A(i,s)$ to be
\begin{gather}
\label{sec6_eq_weight-tableau}
(t_{\gamma,A})_{y_A(i,s)}:= \alpha.
	\end{gather}
That is, for any given $i\in\{1,\dots,n\}$ we let the boxes of the $i$'th horizontal strip be filled, from right to left, by $\gamma_{i,1}$ $1$'s, then $\gamma_{i,2}$ $2$'s, followed by $\gamma_{i,3}$ $3$'s and so on.

As an example consider $n=3$, $p=4$,
\begin{gather}
\label{sec6_eq_sample-A-and-gamma}
\ytableausetup{centertableaux,boxsize=1.1em}
A=\ytableaushort{11122,223,3}
\qquad\text{and}\qquad
\gamma=
\begin{pmatrix}
2	&1	&0 &0	\\
0	&2	&1 &1	\\
0	&0	&2 &0	\\
	\end{pmatrix}\!.
	\end{gather}
In that case
\begin{gather*}
T_{\gamma,A}=
\ytableausetup{centertableaux,boxsize=1.1em}
\begin{ytableau}
 2& 1 & 1 &*(gray) 2 &*(gray) 2 \\
*(gray) 4 &*(gray) 3 & 3 \\
 3
\end{ytableau}
	\end{gather*}
The gray boxes correspond to the $2$nd horizontal strip. A~notable feature of this construction is that $T_{\gamma_{A,D_A},A}=D_A$.

By applying Lemma~\ref{sec6_lemm_permutation-action} we can now calculate $c_A(\gamma)$. This is the content of Proposition~\ref{sec6_prop_coefficient-formula}. For any $A\in\mathbb{Y}_n(p)$ and~$\gamma\in M_n(p)$ with $\mu_\gamma=\mu_A$, let
\begin{gather}
\label{sec6_eq_f-function}
f_{\gamma,A}(\sigma):=
\begin{cases}
(-1)^{N(T^{\sigma,A}_{\gamma,A})}, & \text{if}\quad T^{\sigma,A}_{\gamma,A}\in \mathbb{T}(\lambda_A,p),\\
0, & \text{otherwise},
\end{cases}
\end{gather}
for all $\sigma\in S_{\mu_A}$. Here $T_{\gamma,A}^{\sigma,A}$ is the permutation of~$T_{\gamma,A}$, defined in~\eqref{sec6_eq_weight-tableau}, by $\sigma$ via the action defined in~\eqref{sec6_eq_def-permutation-action}.

\begin{Proposition}\label{sec6_prop_coefficient-formula}
For any $A\in\mathbb{Y}_n(p)$ and~$\gamma\in M_n(p)$ with $\mu_\gamma=\mu_A$,
\begin{gather}
\label{sec6_eq_f-sum}
c_A(\gamma)=
\prod_{\substack{1\leq i \leq n \\ 1\leq \alpha \leq p}}\frac{1}{\gamma_{i,\alpha}!}\sum_{\sigma\in S_{\mu_A}} f_{\gamma,A}(\sigma).
\end{gather}
\end{Proposition}
In~\eqref{sec6_eq_f-sum} the constant preceding the sum can be considered an over-counting factor since
\begin{gather}
\label{sec6_eq_over-counting-factor}
\prod_{\substack{1\leq i \leq n \\ 1\leq \alpha \leq p}}\gamma_{i,\alpha}!=\#\big\{\sigma'\in S_{\mu_A}\colon T_{\gamma,A}^{\sigma',A} =T_{\gamma,A}^{\sigma,A}\big\}.
	\end{gather}
	
To obtain a simple formula for the signs $(-1)^{N(T)}$ we consider the following total order on the coordinates $(k,l)$ of~$\lambda\in \mathcal{P}$
\begin{gather*}
(k,l)<(k',l')\qquad \text{if and only if}\qquad l>l',\qquad \text{or}\qquad l=l'\qquad\text{and}\qquad k<k',
	\end{gather*}
that is
\begin{gather*}
(1,\lambda_1)<\cdots<(\lambda_{\lambda_1}',\lambda_1)<\cdots<(1,1)<\cdots<(\lambda_1',1).
	\end{gather*}
This ordering is motivated by the definition
\begin{gather*}
e_T=e_{t_{(1,\lambda_1)}}\cdots e_{t_{(\lambda_{\lambda_1}',\lambda_1)}} \cdots e_{t_{(1,1)}} \cdots e_{t_{(\lambda_1',1)}},
	\end{gather*}
for all $T\in \mathbb{E}(\lambda,p)$. With it we can write
\begin{gather*}
e_T=(-1)^{\#\{(k,l),(k',l')\in\lambda\colon (k,l)>(k',l')\ \text{and}\ t_{(k,l)}<t_{(k',l')}\}}
 e_1^{(\eta_{\gamma_{A,T}})_1}\cdots e_p^{(\eta_{\gamma_{A,T}})_p},
	\end{gather*}
which implies that
\begin{gather}
\label{sec6_eq_sign-of-young-tableau}
(-1)^{N(T)}=(-1)^{\#\{(k,l),(k',l')\in\lambda\colon (k,l)>(k',l')\ \text{and}\ t_{(k,l)}<t_{(k',l')}\}}.
	\end{gather}

We have now reached a point where any coefficient of the form $c_A(\gamma)$ can be explicitly calculated. To show how such a calculation is done consider $A$ and~$\gamma$ as in~\eqref{sec6_eq_sample-A-and-gamma}, with $n=3$ and~$p=4$. Here $\mu_A=(3,4,2)$ meaning that $S_{\mu_A}$ contains $|S_{\mu_A}|=3!4!2!=288$ different permutations. Taking the over-counting factor~\eqref{sec6_eq_over-counting-factor} into account the total number of Young tableaux generated by the permutation action~$\pi_A$ is
\begin{gather*}
\#\big\{T_{\gamma,A}^{\sigma,A}\colon \sigma\in S_{\mu_A}\big\}=\prod_{\substack{1\leq i \leq n \\ 1\leq \alpha \leq p}}\frac{|S_{\mu_A}|}{\gamma_{i,\alpha}!}
=\frac{3!4!2!}{2^3}=36.
	\end{gather*}
Out of these $36$ tableaux $19$ are not column distinct and thus contribute with factors of~$0$ to~$c_A(\gamma)$, see~\eqref{sec6_eq_f-function}. An example of this is obtained by letting $\tau=(I,(2413),I)$. The corresponding tableau
\begin{gather*}
T_{\gamma,A}^{\tau,A}=\ytableaushort{21134,223,3}
	\end{gather*}
has two entries containing a $2$ in its first column and is thus not column distinct.
The remaining $17$ tableaux are column distinct and each contribute with a factor of~$\pm 1$ to $c_A(\gamma)$. Examples from these $17$ tableaux include the ones corresponding to permutations
\begin{gather*}
\sigma=(I,(13),I) \qquad\text{and}\qquad \kappa=((13),(2413),I).
	\end{gather*}
Those are
\begin{gather*}
T_{\gamma,A}^{\sigma,A}=\ytableaushort{21123,423,3}
\qquad\text{and}\qquad
T_{\gamma,A}^{\kappa,A}=\ytableaushort{11234,223,3}
	\end{gather*}
Calculated with~\eqref{sec6_eq_sign-of-young-tableau} the contributions of these tableaux are
\begin{gather*}
(-1)^{N(T_{\gamma,A}^{\sigma,A})} = (-1)^{11}=-1
\qquad\text{and}\qquad
(-1)^{N(T_{\gamma,A}^{\kappa,A})} = (-1)^{20}=1.
	\end{gather*}
Out of all $17$ column distinct tableaux one will find that $9$ have positive sign and~$8$ have negative meaning that
\begin{gather*}
c_A(\gamma) = 9-8 =1.
	\end{gather*}
	
Obtaining the coefficients $c_A(\gamma)$ in this way requires the construction of all the relevant Young tableaux. Formulas for doing these calculations without explicitly constructing the tableaux become rather complicated. For the interested readers such formulas are included in Appendix~\ref{app2}.

In this section, we have so far explained how to calculate all the coefficients in the expansions from~\eqref{sec6_eq_action-initial-expansion}. Such calculations are done by combining Proposition~\ref{sec6_prop_priliminary-formula} with Proposition~\ref{sec6_prop_coefficient-formula} and Lemma~\ref{sec6_lemm_formulas-matrix-elements}.
Proposition~\ref{sec6_prop_verma-matrix-elements} shows how, using the coefficients from~\eqref{sec6_eq_action-initial-expansion}, we can obtain the expansions for the actions of~$B_i^+$ and~$B_i^-$ on the basis vectors $\tilde{v}_A(p)=\frac{1}{\lambda_A!}v_A(p)$ of~$\overline{V}_n(p)$.
\begin{Proposition}
\label{sec6_prop_verma-matrix-elements}
Let $i\in\{1,\dots,n\}$ and~$A,B\in\mathbb{Y}_n$, then
\begin{gather}
B_i^+\tilde{v}_A(p)=\sum_{B\in \mathbb{Y}_n} \bar{c}_B(i,n,A) \tilde{v}_B(p),\\
B_i^-\tilde{v}_A(p)=\sum_{B\in \mathbb{Y}_n} \big((n+1-p)\hat{c}_B(i,n,A)
	-(n-p)\hat{c}_B(i,n+1,A)\big) \tilde{v}_B(p).
\label{sec6_eq_verma-matrix-elements}
	\end{gather}
	\end{Proposition}

By Theorem~\ref{sec5_theo_basis}, formulas corresponding to~\eqref{sec6_eq_verma-matrix-elements} also hold for the actions of~$X_i$ and~$D_i$ on~$L_n(p)$, though with the sums only being over $B\in\mathbb{Y}_n(p)$.
Proposition~\ref{sec6_prop_verma-matrix-elements} tells us that by using the tools presented in this section to determine matrix coefficients of the $X_i$ and~$D_i$ actions on~$L_n(n)$ and~$L_n(n+1)$ we can determine the matrix coefficients of any other action, be it on~$L_n(p)$ or $\overline{V}_n(p)$, by simple linear combination of coefficients.

\begin{proof}
Let $|0\rangle$ be the lowest weight vector of~$\overline{V}_n(p)$.
Corollary~\ref{sec5_coro_basis-for-creation-subalgebra} tells us that there exist coefficients $d_B(i,A)$, independent of~$p$, such that
\begin{gather*}
B_i^+\tilde{v}_A(p)
=B_i^+\bigg(\frac{1}{\lambda_A!}B_A^+\bigg)|0\rangle=\sum_{B\in\mathbb{Y}_n} d_B(i,A) \frac{1}{\lambda_B!}B_B^+|0\rangle
=\sum_{B\in\mathbb{Y}_n} d_B(i,A)\tilde{v}_B(p).
	\end{gather*}
Since $\Psi_n$ is an isomorphism of vector spaces with $\Psi_n(\tilde{\omega}_B(n))=\tilde{v}_B(n)$, it follows that $d_B(i,A)=\bar{c}_B(i,n,A)$, for all $B\in \mathbb{Y}_n=\mathbb{Y}_n(n)$, when we compare with the coefficients in~\eqref{sec6_eq_action-initial-expansion}.

Using~\eqref{sec2_eq_PCR} we can find $B^{1,+}(i,j,A),B^{2,+}(i,A),B^{3,+}(i,A)\in U(\mathfrak{osp}(1|2n)^+)$ such that
\begin{gather*}
B_i^-B_A^+= \bigg(\sum_{1\leq i\leq j\leq n} B^{1,+}(i,j,A)\big\{B_i^-,B_j^+\big\}\bigg) + B^{2,+}(i,A) + B^{3,+}(i,A)B_i^- .
	\end{gather*}
Corollary~\ref{sec5_coro_basis-for-creation-subalgebra} together with~\eqref{sec2_eq_lw-action-on-vacuum} then tells us that there exists coefficients $d^1_B(i,A)$ and~$d^2_B(i,A)$, independent of~$p$, such that
\begin{gather}
B_i^-\tilde{v}_A(p)
= B_i^-\bigg(\frac{1}{\lambda_A!}B_A^+\bigg)|0\rangle
\\ \hphantom{B_i^-\tilde{v}_A(p)}
= \sum_{B\in\mathbb{Y}_n} d^1_B(i,A)\frac{1}{\lambda_B!}B_B^+\{B_i^-,B_i^+\}|0\rangle + \sum_{B\in\mathbb{Y}_n} d^2_B(i,A)\frac{1}{\lambda_B!}B_B^+|0\rangle
\\ \hphantom{B_i^-\tilde{v}_A(p)}
{}=\sum_{B\in\mathbb{Y}_n} \big(d^1_B(i,A)p + d^2_B(i,A)\big)\tilde{v}_B(p).
\label{sec6_eq_annihilation-verma-expansion-proto}
\end{gather}

The fact that $\Psi_n$ and~$\Psi_{n+1}$ are isomorphisms of vector spaces with $\Psi_n(\tilde{\omega}_B(n))=\tilde{v}_B(n)$ and~$\Psi_{n+1}(\tilde{\omega}_B(n+1))=\tilde{v}_B(n+1)$ lets us compare coefficients with~\eqref{sec6_eq_action-initial-expansion} and obtain the equations
\begin{gather*}
d^1_B(i,A)n + d^2_B(i,A) = \hat{c}_B(i,n,A),
	\end{gather*}
and
\begin{gather*}
d^1_B(i,A)(n+1) + d^2_B(i,A) = \hat{c}_B(i,n+1,A).
	\end{gather*}
Solving these equations for~$d^1_B(i,A)$ and~$d^2_B(i,A)$ and inserting the results into~\eqref{sec6_eq_annihilation-verma-expansion-proto} gives the statement of the proposition.
	\end{proof}

\section[Example: the case n=2]{Example: the case $\boldsymbol{n=2}$}\label{sec7}

In the simplest non-trivial case, that is when $n=2$, the action of~$\mathfrak{osp}(1|4)$ on the tableau vectors is particularly simple. The illustration of this is the subject of this short section.
For any $k,l,m\in\mathbb{N}_0$ we let $A(k,l,m)$ denote the following s.s.\ Young tableau in $\mathbb{Y}_2(p)$
\begin{gather*}
\ytableausetup{aligntableaux=top,boxsize=1.1em}
A(k,l,m):=\underbrace{\ytableaushort{1,2}\cdots\ytableaushort{1,2}}_{m}
	\underbrace{\ytableaushort{1}\cdots\ytableaushort{1}}_{l}
	\underbrace{\ytableaushort{2}\cdots\ytableaushort{2}}_{k}\,.
	\end{gather*}
Any s.s.\ Young tableau from $\mathbb{Y}_2(p)$ can be written in this form. We~note that in this description~$m$ refers to the number of~$\ytableausetup{smalltableaux}\ytableaushort{1,2}$-columns in $A(k,l,m)$ and similarly that $l$ and~$k$ denote the number of~$\ytableausetup{smalltableaux}\ytableaushort{1}$- and~$\ytableausetup{smalltableaux}\ytableaushort{2}$-columns in $A(k,l,m)$ respectively.
The corresponding tableau vector is then
\begin{gather*}
\omega_{A(k,l,m)}(p)= X_2^kX_1^l[X_1,X_2]^m.
	\end{gather*}

The actions of the generators $X_1,X_2,D_1$ and~$D_2$ of~$\mathfrak{osp}(1|4)$ on~$\omega_{A(k,l,m)}(p)$ are then given as follows
\begin{gather}
X_1\omega_{A(k,l,m)}(p)= \omega_{A(k,l+1,m)}(p) + (-1)^{l}[k]_2 \omega_{A(k-1,l,m+1)}(p),
\\[.5ex]
X_2\omega_{A(k,l,m)}(p)= \omega_{A(k+1,l,m)}(p),
\\[.5ex]
D_1\omega_{A(k,l,m)}(p)= (-1)^{k+l}(2m+(2p-4)[m]_2) \omega_{A(k+1,l,m-1)}(p)
\\ \hphantom{D_1\omega_{A(k,l,m)}(p)=}
{}+\big(l+(-1)^{k}\big( (-1)^m p - 1 + 4[m]_2 \big)[l]_2\big) \omega_{A(k,l-1,m)}(p)
\\ \hphantom{D_1\omega_{A(k,l,m)}(p)=}
{}+(-1)^l[k]_2(l-[l]_2)\omega_{A(k-1,l-2,m+1)}(p),
\\[.5ex]
D_2\omega_{A(k,l,m)}(p)= (-1)^{k+l+1}(2m +(2p- 4)[m]_2)\omega_{A(k,l+1,m-1)}(p)
\\ \hphantom{D_2\omega_{A(k,l,m)}(p)=}
{}+\big(k+(2m+p-1)[k]_2\big)\omega_{A(k-1,l,m)}(p),
\label{sec7_eq_explicit-action-p=2}
\end{gather}
where $[k]_2=0$ if $k$ is even and~$[k]_2=1$ if $k$ is odd.

The actions in~\eqref{sec7_eq_explicit-action-p=2} hold for any $p\in\mathbb{N}$.
When $p=1$ we have $[X_1,X_2]=0$, meaning that the tableau vectors reduce to regular monomials when we disregard the, in this case, trivial Clifford algebra part. That is, $\omega_{A(k,l,0)}(1)=x_{2,1}^kx_{1,1}^l$ and~$\omega_{A(k,l,m)}(1)=0$, for all $k,l\in\mathbb{N}_0$ and~$m>0$. The actions of the $\mathfrak{osp}(1|4)$ generators similarly reduce to those of usual variable multiplication and differentiation.

\appendix

\section{Graded lexicographic order}
\label{app1}

The definitions of the graded lexicographic ordering on~$\mathbb{N}_0^n$ and~$\mathcal{P}$ were omitted in the main text. For the readers unfamiliar with these orderings the definitions are included here.

\begin{Definition}
Given $\mu,\nu\in\mathbb{N}_0^n$ we say that $\mu<\nu$ with respect to the graded lexicographic ordering if $|\mu|=\sum_{i=1}^n\mu_i < |\nu| =\sum_{i=1}^n\nu_i,$
or if $|\mu|=|\nu|$ and~$\mu_j<\nu_j$ for the first $j$, where $\mu_j$ and~$\nu_j$ differ.
	\end{Definition}
For $n=4$, the $\mu \in\mathbb{N}_0^4$ with $|\mu|\leq 2$ are ordered as follows,
\begin{gather*}
(0,0,0,0)<(0,0,0,1)<(0,0,1,0)<(0,1,0,0)<(1,0,0,0)<(0,0,0,2)
\\ \hphantom{(0,0,0,0)}
{}<(0,0,1,1)<(0,0,2,0)<(0,1,0,1)<(0,1,1,0)<(0,2,0,0)
\\ \hphantom{(0,0,0,0)}
{}<(1,0,0,1)<(1,0,1,0)<(1,1,0,0)<(2,0,0,0).
\end{gather*}

\begin{Definition}
Given $\lambda,\kappa\in \mathcal{P}$ we say that $\lambda<\kappa$ with respect to the graded lexicographic ordering if $|\lambda|<|\kappa|$, or if $|\lambda|=|\kappa|$ and~$\lambda_j<\kappa_j$ for the first $j$, where $\lambda_j$ and~$\kappa_j$ differ.
	\end{Definition}
The $\lambda\in\mathcal{P} \in\mathbb{N}_0^n$ with $|\lambda|\leq 4$ are ordered as follows,
\begin{gather*}
(0)\,{<}\,(1)\,{<}\,(1,1)\,{<}\,(2)\,{<}\,(1,1,1)\,{<}\,(2,1)\,{<}\,(3)\,{<}\,(1,1,1,1)\,{<}
\,(2,1,1)\,{<}\,(2,2)\,{<}\,(3,1)\,{<}\,(4).
	\end{gather*}

\section[Alternative formula for cA(gamma)]
{Alternative formula for~$\boldsymbol{c_A(\gamma)}$}\label{app2}

We present here a formula for calculating the coefficients $c_A(\gamma)$ of~$\tilde{\omega}_A(p)$ that is more explicit than the one produced at the end of Section~\ref{sec6}. What follows are two technical lemmas in which constituents of the final formula are obtained. Following those lemmas the final formula is presented, marking the end of this appendix. Due to them being technical and not very enlightening the proofs of the following results are omitted.

Consider the function
\begin{gather*}
G(l,k,l')=
\begin{cases}
l'-l, & \text{if}\quad l\leq l'\leq k,
\\
l'-k, & \text{if}\quad l\leq k\leq l',
\\
0, & \text{otherwise},
\end{cases}
\end{gather*}
for all $l,k,l'\in \mathbb{N}_0$.

\begin{Lemma}
Given $A\in\mathbb{Y}_n(p)$ and~$\gamma\in M_n(p)$ with $\mu_\gamma=\mu_A$, then
\begin{gather*}
e_{T_{\gamma,A}}=(-1)^{N_{1}(\gamma,A)+N_{2}(\gamma,A)} e^{\eta_\gamma},
\end{gather*}
where
\begin{gather*}
N_{1}(\gamma,A)=\sum_{\substack{1\leq \alpha<\alpha'\leq p, \\ 1\leq i< i'\leq n}}
\ \sum_{s=1+\sum_{\beta=1}^{\alpha-1}\gamma_{i,\beta}}^{\sum_{\beta=1}^{\alpha}\gamma_{i,\beta}}
G\Bigg(\sum_{\beta=1}^{\alpha'-1}\gamma_{i',\beta},\sum_{l=l_A(i,s)+1}^{(\lambda_{A^{i'}})_1}
(\lambda_{A^{i'}})'_l-(\lambda_{A^{i'-1}})'_l,\sum_{\beta=1}^{\alpha'} \gamma_{i',\beta}\Bigg)
\end{gather*}
and
\begin{gather*}
N_{2}(\gamma,A)=\sum_{\substack{1\leq \alpha<\alpha'\leq p, \\ 1\leq i'< i\leq n}}
\ \sum_{s=1+\sum_{\beta=1}^{\alpha-1}\gamma_{i,\beta}}^{\sum_{\beta=1}^{\alpha}\gamma_{i,\beta}}
G\Bigg(\sum_{\beta=1}^{\alpha'-1}\gamma_{i',\beta},\sum_{l=l_A(i,s)}^{(\lambda_{A^{i'}})_1}
(\lambda_{A^{i'}})'_l-(\lambda_{A^{i'-1}})'_l,\sum_{\beta=1}^{\alpha'} \gamma_{i',\beta}\Bigg).
\end{gather*}
\end{Lemma}

Let $k\in\mathbb{N}$ and consider a $k$-tuple of positive integers $L=(L_1,\dots,L_k)\in\mathbb{N}^k$. If $L_m\neq L_{m'}$, for all $m\neq m'$, then we define $\sigma_L\in S_k$ to be the permutation such that
\begin{gather*}
L_{\sigma(1)}<\cdots<L_{\sigma(k)}.
	\end{gather*}
With this we can define the sign of~$L$ to be
\begin{gather*}
\operatorname{sgn}(L)=
\begin{cases}
\operatorname{sgn}(\sigma_L), & \text{if}\quad L_m\neq L_{m'},\ \text{ for all } m\neq m',
\\
0, & \text{otherwise}.
\end{cases}
	\end{gather*}

\begin{Lemma}
Given $A\in \mathbb{Y}_n(p)$, $\sigma\in S_{\mu_A}$ and~$\gamma\in M_{n,p}(\mathbb{N}_0)$ with $\mu_\gamma=\mu_A$, then
\begin{gather*}
\operatorname{sgn}(\sigma)\prod_{\alpha=1}^p\operatorname{sgn}(L_{\gamma,A}(\sigma,\alpha))(-1)^{N_0(\gamma,A)} e_{T_{\gamma,A}}=
\begin{cases}
e_{T_{\gamma,A}^{\sigma,A}}, & \text{if}\quad T_{\gamma,A}^{\sigma,A}\in \mathbb{T}(\lambda_A,p),
\\
0, & \text{otherwise},
\end{cases}
\end{gather*}
where
\begin{gather*}
N_0(\gamma,A)=\sum_{\substack{1\leq \alpha\leq p, \\ 1\leq i< i'\leq n}}
\ \sum_{s=1+\sum_{\beta=1}^{\alpha-1}\gamma_{i,\beta}}^{\sum_{\beta=1}^{\alpha}\gamma_{i,\beta}}
G\Bigg(\sum_{\beta=1}^{\alpha-1}\gamma_{i',\beta},\sum_{l=l_A(i,s)+1}^{(\lambda_{A^{i'}})_1}
(\lambda_{A^{i'}})'_l-(\lambda_{A^{i'-1}})'_l,\sum_{\beta=1}^{\alpha} \gamma_{i',\beta}\Bigg)
\end{gather*}
and~$L_{\gamma,A}(\sigma,\alpha)\in\mathbb{N}^{(\eta_\gamma)_\alpha}$ defined such that, for all $\alpha\in\{1,\dots,p\}$, $i\in\{1,\dots,n\}$ and~$t$ with $\sum_{j=1}^{i-1}\gamma_{j,\alpha}<t\leq\sum_{j=1}^{i}\gamma_{j,\alpha}$,
\begin{gather*}
\big(L_{\gamma,A}(\sigma,\alpha)\big)_t= l_A\Bigg(i,\sigma_i^{-1}\Bigg(t-\sum_{j=1}^{i-1}\gamma_{j,\alpha}+\sum_{\beta=1}^{\alpha-1} \gamma_{i,\beta}\Bigg)\!\Bigg).
\end{gather*}
If $(\eta_\gamma)_\alpha=0$, then $\operatorname{sgn}(L_{\gamma,A}(\sigma,\alpha)):=1$.
	\end{Lemma}

\begin{Proposition}
Let $A\in \mathbb{Y}_n(p)$ and~$\gamma\in M_{n,p}(\mathbb{N}_0)$ with $\mu_\gamma=\mu_A$, then
\begin{gather*}
c_A(\gamma)=\sum_{\sigma\in S_{\mu_A}}\prod_{\substack{1\leq i \leq n \\ 1\leq \alpha \leq p}}\frac{1}{\gamma_{i,\alpha}!}
	\operatorname{sgn}(\sigma)(-1)^{N(\gamma,A)}\operatorname{sgn}(L_A(\sigma,\alpha)),
	\end{gather*}
where $N(\gamma,A)=N_0(\gamma,A)+N_1(\gamma,A)+N_2(\gamma,A)$.
	\end{Proposition}

\subsection*{Acknowledgements}
The authors were supported by the EOS Research Project 30889451. The editor and referees are thanked for their helpful reports.

\pdfbookmark[1]{References}{ref}
\LastPageEnding

\end{document}